\documentclass[11pt]{amsart}
\usepackage[leqno]{amsmath}
\usepackage{amssymb}
\usepackage{enumerate}
\usepackage{mathabx}

\usepackage{subfigure}
\usepackage{graphicx}
\usepackage{bm}
\usepackage[all]{xy}
\usepackage{mathrsfs} 

\usepackage[usenames,dvipsnames]{xcolor}
\usepackage[
colorlinks=true,linkcolor=NavyBlue,urlcolor=RoyalBlue,citecolor=PineGreen,bookmarks=true,bookmarksopen=true,bookmarksopenlevel=2,unicode=true,linktocpage]{hyperref}

\usepackage{microtype}
\usepackage{centernot}
\usepackage{stmaryrd}
\usepackage{comment}
\usepackage{bm}

\numberwithin{equation}{section}

\newcommand{\coloneqq}{:=}

\newcommand{\E}{\mathbf{E}}
\newcommand{\p}{\mathbf{P}}
\newcommand{\UH}{\mathbf{H}}
\newcommand{\C}{\mathbf{C}}
\newcommand{\R}{\mathbf{R}}
\newcommand{\D}{\mathbf{D}}
\newcommand{\N}{\mathbf{N}}

\newcommand{\dist}{\mathrm{dist}}
\newcommand{\SLE}{\mathrm{SLE}}
\newcommand{\CLE}{\mathrm{CLE}}

\renewcommand{\epsilon}{\varepsilon}

\newcommand{\SSS}{\mathscr{S}}
\newcommand{\G}{\mathscr{G}}
\newcommand{\F}{\mathscr{F}}
\newcommand{\M}{\mathfrak{M}}
\newcommand{\wt}{\widetilde}
\newcommand{\wh}{\widehat}
\newcommand{\wc}{\widecheck}
\newcommand{\cont}{\mathrm{Cont}}

\newcommand{\re}{\mathrm{Re}}
\newcommand{\im}{\mathrm{Im}}
\newcommand{\FH}{\mathfrak{H}}
\newcommand{\area}{\mathrm{Area}}
\newcommand{\CM}{\mathcal{M}}
\newcommand{\Fh}{\mathfrak{h}}
\newcommand{\one}{\mathbf{1}}
\newcommand{\SB}{\mathscr{B}}
\newcommand{\CG}{\mathcal{G}}

\newcommand{\diam}{\mathrm{diam}}
\newcommand{\CQ}{\mathcal{Q}}

\newcommand{\Po}{\mathrm{Poisson}}
\newcommand{\hcap}{\mathrm{hcap}}

\newcommand{\CS}{{\mathcal S}}
\newcommand{\CL}{{\mathcal L}}

\newtheorem{thm}{Theorem}[section]
\newtheorem{lem}[thm]{Lemma}
\newtheorem{prop}[thm]{Proposition}

\theoremstyle{definition}
\newtheorem{defin}[thm]{Definition}
\newtheorem{rmk}[thm]{Remark}

\newcommand{\giv}{\,|\,}

\newcommand{\Z}{{\mathbf Z}}

\newcommand{\ol}{\overline}
\newcommand{\ul}{\underline}

\begin{document}

\title[Existence and uniqueness of the conformally covariant volume measure on CLEs]{Existence and uniqueness of the conformally covariant volume measure on conformal loop ensembles}

\author{Jason Miller and Lukas Schoug}

\begin{abstract}
We prove the existence and uniqueness of the canonical conformally covariant volume measure on the carpet/gasket of a conformal loop ensemble ($\CLE_\kappa$, $\kappa \in (8/3,8)$) which respects the Markov property for $\CLE$.  The starting point for the construction is the existence of the canonical measure on $\CLE$ in the context of Liouville quantum gravity (LQG) previously constructed by the first author with Sheffield and Werner.  As a warm-up, we construct the natural parameterization of $\SLE_\kappa$ for $\kappa \in (4,8)$ using LQG which serves to complement earlier work of Benoist on the case $\kappa \in (0,4)$.
\end{abstract}

\date{\today}
\maketitle

\parindent 0 pt
\setlength{\parskip}{0.20cm plus1mm minus1mm}

\setcounter{tocdepth}{1}
\tableofcontents

\section{Introduction}
\label{sec:intro}

The \emph{conformal loop ensembles} ($\CLE_\kappa$, $\kappa \in (8/3,8)$) are the canonical conformally invariant probability measures on locally finite collections of non-crossing loops in a simply connected domain $D \subseteq \C$ \cite{sheffield2009exploration,sw2012cle}.  It is the loop version of the Schramm-Loewner evolution ($\SLE_\kappa$) \cite{s2000sle}.  Just like $\SLE_\kappa$, $\CLE_\kappa$ arises in the context of two-dimensional statistical mechanics as the scaling limit of the collection of all interfaces of certain models at criticality \cite{s2001perc,cn2006fullpercolation,lsw2004lerwust,s2010ising,bh2019ising,ks2019ising}.  The $\CLE_\kappa$'s have the same phases as the $\SLE_\kappa$'s \cite{rs2005basic} in that when $\kappa \in (8/3,4]$ the loops are a.s.\ simple, disjoint, and do not intersect the domain boundary and if $\kappa \in (4,8)$ they are a.s.\ self-intersecting, intersect each other and the domain boundary. The set of points which are not surrounded by a loop in the $\CLE_\kappa$ is called the $\CLE_\kappa$ \emph{carpet} if $\kappa \in (8/3,4]$ and the $\CLE_\kappa$ \emph{gasket} if $\kappa \in (4,8)$.  The reason for this terminology is that in the former (resp.\ latter) phase a $\CLE_\kappa$ is a random analog of the Sierpinski carpet (resp.\ gasket).  It is proven (see \cite{ssw2009confrad,nw2011soupsdim,msw2014gasketdim}) that the dimension of the $\CLE_\kappa$ carpet or gasket is a.s.\ 
\begin{align}
\label{eq:CLEdim}
    d(\kappa) = 2-\frac{(8-\kappa)(3\kappa-8)}{32\kappa} = 1+\frac{2}{\kappa} + \frac{3\kappa}{32} \quad\text{for}\quad \kappa \in (8/3,8).
\end{align}
Note that $d(\kappa) = 2$ at the boundary points $\kappa=8/3,8$ and that $d(\kappa) < 2$ for $\kappa \in (8/3,8)$.  The reason for the former is that a $\CLE_{8/3}$ can be interpreted as being the empty set of loops while a $\CLE_8$ consists of a single space-filling loop.

In this paper, we will construct a measure on the $\CLE_\kappa$ carpet or gasket for each $\kappa \in (8/3,8)$ which is determined by the $\CLE_\kappa$, conformally covariant, and respects the Markov property of $\CLE_\kappa$ and we will prove that this measure is unique up to a multiplicative constant.  We will now introduce what we mean precisely by this.  In doing so, we will make use of the following notation.  Suppose that $\Gamma$ is $\CLE_\kappa$ on a simply connected domain and $U \subseteq D$ is open.  We let $U^\Gamma$ denote the set obtained by removing from~$U$ the closure of the union of the set of points surrounded by those loops of~$\Gamma$ which intersect $U$ and $D \setminus U$.  For $V \subseteq D$ open, we let~$\Gamma_V$ be the loops of~$\Gamma$ which are contained in~$\ol{V}$.

\begin{defin}
\label{def:cov_measure}
Fix $\kappa \in (8/3,8)$.  Suppose that for each simply connected domain $D$ we have a probability measure $\p_D$ on pairs $(\Gamma,\Xi)$ where the marginal law of $\Gamma$ is that of a $\CLE_\kappa$ on $D$ and $\Xi$ is a measure supported on its carpet (if $\kappa \leq 4$) or gasket (if $\kappa > 4$) $\Upsilon$ which is a.s.\ determined by $\Gamma$.
\begin{enumerate}[(i)]
\item\label{it:conf_cov} Then we say that the family $(\p_D)$ is \emph{conformally covariant with exponent $d$} if the following is true.  Suppose that $\varphi \colon D \to \wt{D}$ is a conformal transformation and $(\Gamma,\Xi) \sim \p_D$, $(\wt{\Gamma},\wt{\Xi}) \sim \p_{\wt{D}}$ are coupled together so that $\wt{\Gamma} = \varphi(\Gamma)$.  Then for all Borel sets $A \subseteq D$ we have that
\[ \wt{\Xi}(\varphi(A)) = \int_A |\varphi'(z)|^d d\Xi(z).\]
\item\label{it:markov} If $(\p_D)$ satisfies~\eqref{it:conf_cov} then we say that $(\p_D)$ \emph{respects the $\CLE_\kappa$ Markov property} if the following is true.  Suppose that $(\Gamma,\Xi) \sim \p_D$, $U \subseteq D$ is open and simply connected, and $V$ is a component of $U^\Gamma$.  Then the conditional law of $(\Gamma_V,\Xi|_V)$ given $\Gamma \setminus \Gamma_V$ and $\Xi|_{D \setminus V}$ is $\p_V$.
\item\label{it:locally_finite_expectation} For each compact set $K \subseteq D$ we have that $\E_D[ \Xi(K) ] < \infty$ where $\E_D$ denotes the expectation under $\p_D$.
\end{enumerate}
\end{defin}

Our main result is the following.
 
\begin{thm}
\label{thm:mainresult}
For each $\kappa \in (8/3,8)$ there exists a family $(\p_D)$ of probability measures which satisfy Definition~\ref{def:cov_measure} with exponent $d=d(\kappa)$ as in~\eqref{eq:CLEdim} so that if $(\Gamma,\Xi) \sim \p_D$ then a.s.\ $\Xi(D) > 0$.  Moreover, if $(\wt{\p}_D)$ is another family satisfying Definition~\ref{def:cov_measure} then there exists a constant $K \geq 0$ so that if $(\Gamma,\Xi) \sim \p_D$ then $(\Gamma, K \Xi) \sim \wt{\p}_D$ for every simply connected domain $D \subseteq \C$.
\end{thm}

\begin{rmk}
By \cite[Lemma~2.6]{zhan2021loop}, we know that if the $d(\kappa)$-dimensional Minkowski content $\cont_\Upsilon$ of the carpet or gasket of $\Gamma \sim \CLE_\kappa$ exists (usually just referred to as the Minkowski content of $\Gamma$) and has locally finite expectation, then it satisfies Definition~\ref{def:cov_measure}.  Hence, by the uniqueness component of Theorem~\ref{thm:mainresult}, we have that if $\cont_\Upsilon$ exists, is non-trivial, and has locally finite expectation then there exists a deterministic constant $K > 0$ so that it equals $K\Xi$.
\end{rmk}

We construct this family of measures using as a starting point the natural measure on $\CLE_\kappa$ in the context of Liouville quantum gravity (LQG) which has been constructed for $\kappa \in (8/3,8) \setminus \{4\}$ \cite{msw2020simpleclelqg,msw2020nonsimpleclelqg}. The case $\kappa = 4$ is slightly different in the LQG context as it corresponds to the critical case $\gamma = 2$ which is why it is not treated in \cite{msw2020simpleclelqg,msw2020nonsimpleclelqg}.  We will construct the measure in this case by considering a sequence $(\kappa_n)$ with $\kappa_n < 4$ such that $\kappa_n \nearrow 4$ and define the measure as the weak limit as $n \to \infty$ of the measure on the $\CLE_{\kappa_n}$ carpet and show that it is conformally covariant with exponent $d(4)$ and is supported on the $\CLE_4$ carpet.

Before constructing the measures on $\CLE$, we construct the natural parameterization of $\SLE_{\kappa'}$ for $\kappa' \in (4,8)$ using the GFF.  
The natural parameterization was constructed in \cite{ls2011natural,lz2013natural} and is conjectured to be the parameterization of an $\SLE$ which arises as a scaling limit of an interface of a discrete model, where the discrete curve is parameterized by the number of edges which it traverses. It is the unique locally finite measure on $\SLE$ satisfying a certain conformal covariance (see Section~\ref{sec:preliminaries_natural_parameterization}). Our second main result is the following (for the definitions of the various objects, see Section~\ref{sec:prel}).
\begin{thm}\label{thm:mainresult2}
Let $\eta \sim \SLE_{\kappa'}$, $\kappa' \in (4,8)$, let $h$ be a zero-boundary GFF independent of $\eta$ and let $\sigma^{\eta,h}$ denote the generalized quantum length of $\eta$ with respect to $h$. Set
\begin{align*}
	\mu(dz) = F(z) \E\!\left[ \sigma^{\eta,h}(dz) \, \middle| \, \eta \right],
\end{align*}
where $F(z) = r_{\UH}(z)^{-\kappa'/8}$ and $r_{\UH}(z)$ is the conformal radius of $\UH$ seen from $z$. Then, there is a deterministic constant $C>0$ such that $C\mu$ is a.s.\ the natural parameterization of $\eta$.
\end{thm}
It follows from~\cite{ls2011natural} that if $\mu$ is an a.s.\ locally finite measure on $\eta$ which satisfies a certain conformal covariance formula, then $\mu$ must be a (deterministic) constant times the natural parameterization of $\eta$, see Section~\ref{sec:preliminaries_natural_parameterization}.  Thus, Section~\ref{sec:npSLE} is devoted to proving conformal covariance and local finiteness of $\mu$ and hence Theorem~\ref{thm:mainresult2}. This also makes for a good warm-up for proving Theorem~\ref{thm:mainresult}, as the conformal covariance of the measure on $\CLE_\kappa$ is proven similarly, but with more cases and the transformation of quantum length is a bit more subtle in places. Moreover, in proving that $\mu$ is locally finite, we determine the density of the intensity $\E[\mu]$, and the strategy for doing this is in some aspects similar to the strategy for determining the density of the intensity $\E[\Xi]$ (where $\Xi$ is a measure on $\CLE_\kappa$ as in Definition~\ref{def:cov_measure}) which is an important step in proving the uniqueness part of Theorem~\ref{thm:mainresult}.

\subsection*{Related work} In recent years, there have been a number of works focused on constructing natural measures supported on random fractals. In \cite{ls2011natural,lr2015minkowski} the so-called natural parameterization of $\SLE$ is studied and shown to be the Minkowski content of $\SLE$. In \cite{benoist2018natural}, the same measure is constructed for $\kappa \in (0,4)$ via a different method: as the conditional expectation of a certain LQG length measure on the $\SLE_\kappa$ process. This is the strategy we will take in constructing the measures in Sections~\ref{sec:npSLE} and~\ref{sec:ccCLE}. Another related object is a random measure on the so-called two-valued sets of the Gaussian free field (GFF), which was constructed in \cite{ssv2020tvsdim}. It was constructed similarly to the above, but using the imaginary multiplicative chaos (i.e., LQG but with imaginary parameter) and it was shown that if the conformal Minkowski content (which is defined by replacing the Euclidean distance with the conformal radius of the domain minus the fractal in the definition of Minkowski content) exists, then it is equal to the constructed measure.

\subsection*{Outline}

We now give a brief overview of the paper. In Section~\ref{sec:prel} we introduce the preliminary material needed: the Schramm-Loewner evolution, conformal loop ensembles, GFFs and LQG. Section~\ref{sec:npSLE} is devoted to the problem of constructing the natural parameterization of $\SLE_\kappa$, for $\kappa \in (4,8)$ using LQG. In Section~\ref{sec:ccCLE} we prove the existence of the conformally covariant measures on the $\CLE$ carpets/gaskets and Section~\ref{sec:properties} is devoted to proving the uniqueness.

\subsection*{Acknowledgments}

J.M. and L.S.\ were supported by the ERC starting grant 804166 (SPRS). The authors would like to thank an anonymous referee for helpful comments.

\section{Preliminaries}\label{sec:prel}
\subsection{Random measures}
We begin by briefly recalling some notation and definitions related to random measures and we refer the reader to~\cite{kallenberg2017rmbook} for more details. A random measure is a random element taking values in a space of measures on some Borel space. Here we consider measures on~$\C$. Let $\xi$ be a random measure. Then the \emph{intensity} of $\xi$ is the measure $\E[\xi]$ which is defined by $\E[\xi](O) = \E[\xi(O)]$ for each Borel set $O \subseteq \C$.  Furthermore, if $\G$ is a $\sigma$-algebra then the \emph{conditional intensity} is the random element $\E[\xi \, | \, \G]$ defined by $\E[\xi \, | \, \G](O) = \E[ \xi(O) \, | \, \G]$.  By \cite[Corollary~2.17]{kallenberg2017rmbook} we have that if $\xi$ is an a.s.\ locally finite measure and $\E[\xi]$ is locally finite, then there is a version of $\E[\xi \, | \, \G]$ which is an a.s.\ locally finite measure. Finally, for a function $f : \C \to \R$ we write $\xi(f) = \int_{\C} f(z) d\xi(z)$. We remark, that our definition of random measure (and hence the related notions) differs slightly from that of~\cite{kallenberg2017rmbook}, in that we do not require that a random measure is a.s.\ locally finite. We will however prove that the random measures that we are working with are in fact a.s.\ locally finite, and thereafter there is no problem in using the results in~\cite{kallenberg2017rmbook}.

\subsection{Schramm-Loewner evolution}

We will now introduce the Schramm-Loewner evolution ($\SLE$). For more on $\SLE$, see e.g., \cite{lawler2005slebook,rs2005basic}. Fix $\kappa > 0$ and for each $z \in \UH$ let $g_t(z)$ denote the solution to
\begin{align}
\label{eq:LDE}
    \partial_t g_t(z) = \frac{2}{g_t(z)-W_t}, \quad g_0(z)=z,
\end{align}
where $W_t = \sqrt{\kappa} B_t$ and $B_t$ is a standard Brownian motion. The solution to~\eqref{eq:LDE} exists until the time $T_z = \inf\{t >0: \im(g_t(z)) = 0 \}$ and we write $K_t = \{z \in \UH: T_z \leq t \}$.  Then $(g_t)_{t \geq 0}$ is a family of conformal maps called the $\SLE_\kappa$ Loewner chain and $g_t : \UH \setminus K_t \to \UH$ is the unique conformal map with $g_t(z) - z \to 0$ as $z \to \infty$.  It was proved by Rohde and Schramm for $\kappa \neq 8$ \cite{rs2005basic} and by Lawler, Schramm, and Werner for $\kappa = 8$ \cite{lsw2004lerwust} that there a.s.\ exists a continuous curve $\eta$ in $\ol{\UH}$ from $0$ to $\infty$ such that $\UH \setminus K_t$ is the unbounded component of $\UH \setminus \eta([0,t])$.  The curve $\eta$ is an $\SLE_\kappa$ process from $0$ to $\infty$ in $\UH$. The behavior of $\SLE_\kappa$ processes depends heavily on $\kappa$: when $\kappa \leq 4$, they are a.s.\ simple, when $\kappa \in (4,8)$ they have self-intersections and intersect the domain boundary and when $\kappa \geq 8$ they are a.s.\ space-filling \cite{rs2005basic}.

The range of an $\SLE_\kappa$ process is a random fractal of a.s.\ Hausdorff dimension $\min(2,1+\frac{\kappa}{8})$ \cite{rs2005basic,v2008dim} which is scale invariant and satisfies the so-called \emph{domain Markov property}: if $\tau$ is an a.s.\ finite stopping time then the law of $t \mapsto g_\tau(\eta(\tau+t))-W_\tau$ is that of an $\SLE_\kappa$ process in $\UH$ from $0$ to $\infty$. One defines $\SLE_\kappa$ processes in other simply connected domains as the conformal images of $\SLE_\kappa$ in~$\UH$.

The so-called $\SLE_\kappa(\ul{\rho})$ processes \cite[Section~8.3]{lsw2003confres} are a natural generalization of $\SLE_\kappa$ processes where one keeps track of several marked points. More precisely, consider $\underline{z} = (z_1,\dots,z_n)$ where $z_j \in \overline{\UH}$ and let $\underline{\rho} = (\rho_1,\dots,\rho_n)$ where $\rho_j \in \R$. Then the $\SLE_\kappa(\underline{\rho})$ process with force points $\underline{z}$ and weights $\ul{\rho}$ arises from the family of conformal maps in the same way as ordinary $\SLE_\kappa$ by solving~\eqref{eq:LDE} with $W_t$ given by the solution to the SDE
\begin{align*}
    dW_t &= \sqrt{\kappa}dB_t + \sum_{j=1}^n \re\!\left( \frac{\rho_j}{W_t - V_t^j} \right)dt,\quad
    dV_t^j =  \frac{2}{V_t^j - W_t} dt, \quad V_0^j = z_j.
\end{align*}
The existence of the curve corresponding to an $\SLE_\kappa(\ul{\rho})$ process was proved in \cite{ms2016ig1} up until the \emph{continuation threshold} which is
\[ \inf\left\{ t \geq 0 : \sum_{i : V_t^i = W_t} \rho_i \leq -2 \right\},\]
i.e., the first time $t$ that the sum of the weights of the force points which collide with $W$ at time $t$ is at most $-2$.  The continuity of the $\SLE_\kappa(\rho)$ processes (with a single force point) with $\rho < -2$ was proved in \cite{ms2019lightcones,msw2017clepercolation}.  When they are not interacting with their force points, $\SLE_\kappa(\underline{\rho})$ processes locally look like ordinary $\SLE_\kappa$ processes and in this case can be obtained by weighting the law of an ordinary $\SLE_\kappa$ by a certain martingale; see \cite[Theorem~6]{sw2005coordinate}. For more on $\SLE_\kappa(\underline{\rho})$ processes see \cite{ms2016ig1}.

\subsection{Natural parameterization of $\SLE$}\label{sec:preliminaries_natural_parameterization}

Implicit in the definition of $\SLE_\kappa$ using the chordal Loewner equation~\eqref{eq:LDE} is a notion of time called the capacity parameterization.  That is, $\hcap(K_t) = 2t$ for all $t \geq 0$ where $\hcap$ denotes the half-plane capacity.  This parameterization of time is convenient in the context of the Loewner equation but there are other time parameterizations for $\SLE_\kappa$ which are useful to study for other reasons.  One such example is the \emph{natural parameterization} of $\SLE_\kappa$, which was first considered in~\cite{ls2011natural}.  The natural time parameterization is the one which conjecturally corresponds to the scaling limit of an interface for a discrete model where the discrete curve is parameterized by the number of edges which it traverses.  (So far this has only been proved in the case of the convergence of the loop-erased random walk~\cite{lv2021natural}.)   The natural parameterization for $\SLE_\kappa$ was constructed in~\cite{ls2011natural} for $\kappa < 4(7 - \sqrt{33})$ and later in~\cite{lz2013natural} for all $\kappa < 8$.  For $\kappa \geq 8$, the natural parameterization corresponds to parameterizing the curve by Lebesgue measure as it is space-filling.  The construction of the natural parameterization in~\cite{ls2011natural,lz2013natural} is indirect.  A direct construction was first given in~\cite{lr2015minkowski}, in which it is shown that it is (a constant times) the $d=(1+\kappa/8)$-dimensional Minkowski content of $\eta \sim \SLE_\kappa$.  That is, let $\CM(D) \coloneqq \cont_d(\eta \cap D)$, where
\begin{align*}
    \cont_d(A) = \lim_{r \to 0} r^{d-2} \area(\{z \in \UH: \dist(z,A) < r \}).
\end{align*}
If $\eta$ is has the natural parameterization, then $\CM(\eta([s,t])) = t-s$ for all $0 \leq s \leq t$. Let $(f_t)_{t\geq 0}$ be the centered Loewner chain for $\eta$, i.e., $f_t(z) = g_t(z) - \sqrt{\kappa} B_t$ where $\sqrt{\kappa} B_t$ is the driving function of $\eta$.  If $D \subseteq \UH$ is a subdomain with piecewise smooth boundary, then $\CM$ satisfies in addition the following.
\begin{align*}
    &\E[\CM(D)] = \int_D G_\kappa(z) dz, \quad \E[\CM(D)^2] = \int_D \int_D G_\kappa(z,w) dz dw, \\
    &\E[ \CM(D) \, | \, \eta|_{[0,t]} ] = \CM|_{\eta([0,t])}(D) + \int_D |f_t'(z)|^{2-d} G_\kappa(f_t(z)) dz,
\end{align*}
where $G_\kappa(z) = c \sin^{\frac{8}{\kappa}-1}(\arg z) \im(z)^{d-2}$ for some $c = c_\kappa>0$ and
\begin{align*}
    G_\kappa(z,w) = \lim_{r \to 0} r^{2(d-2)} \p[\dist(z,\eta) \leq r, \dist(w,\eta) \leq r].
\end{align*}
The natural parameterization of $\SLE_\kappa$ is characterized by the property that it is conformally covariant and respects the domain Markov property which we will now explain.  We write $\psi_t = f_t^{-1}$ and note that the map $f_t$ ``unzips'' $t$ units of time from the curve $\eta$, while $\psi_t$ ``zips up'' the same amount of time.  We define the unzipped curves by
\begin{align}
\label{eq:unzipped_curve}
    \eta^t(u) = f_t(\eta(t+u)) \quad\text{for}\quad u \geq 0,
\end{align}
and write $\eta^0 = \eta$.  For each $t \geq 0$, let
\begin{align}
\label{eq:confcovSLE}
    \mu^t(dz) = |\psi_t'(z)|^{-d} \mu \circ \psi_t(dz).
\end{align}
The following was proven (but formulated a bit differently) in~\cite{ls2011natural}.
\begin{thm}
\label{thm:SLE_unique}
Fix $\kappa \in (0,8)$, set $d = 1+\kappa/8$, and let $\eta$ be an $\SLE_\kappa$ process and let $\mu$ be an a.s.\ locally finite measure on $\eta$.  If $(\eta^t,\mu^t) \stackrel{d}{=} (\eta,\mu)$ for each $t \geq 0$ then there exists a constant $C > 0$ so that $\mu$ is $C$ times the natural parameterization of $\eta$.
\end{thm}
We say that a measure $\mu$ determined by $\eta$ is conformally covariant with exponent $d$ and respects the domain Markov property if the hypothesis given in Theorem~\ref{thm:SLE_unique} holds.

As mentioned in the introduction, the natural parameterization of $\SLE_\kappa$, $\kappa < 4$ can also be constructed via Gaussian multiplicative chaos measures; see~\cite{benoist2018natural}. In Section~\ref{sec:npSLE}, we shall do the same but in the case of self-intersecting $\SLE_\kappa$, that is, for $\kappa \in (4,8)$.

\subsection{Conformal loop ensembles}
\label{sec:CLE}

As mentioned earlier, conformal loop ensembles ($\CLE_\kappa$, $\kappa \in (8/3,8)$) are a one-parameter family of collections of loops that locally look like $\SLE$ curves, constructed in \cite{sheffield2009exploration,sw2012cle}. $\SLE_\kappa$ arises naturally when studying the scaling limit of a single interface in a statistical mechanics model with certain boundary conditions while $\CLE_\kappa$ arises as the scaling limit of all the interfaces simultaneously.

\begin{figure}[h!]
	\centering
		\includegraphics[width=0.32\textwidth]{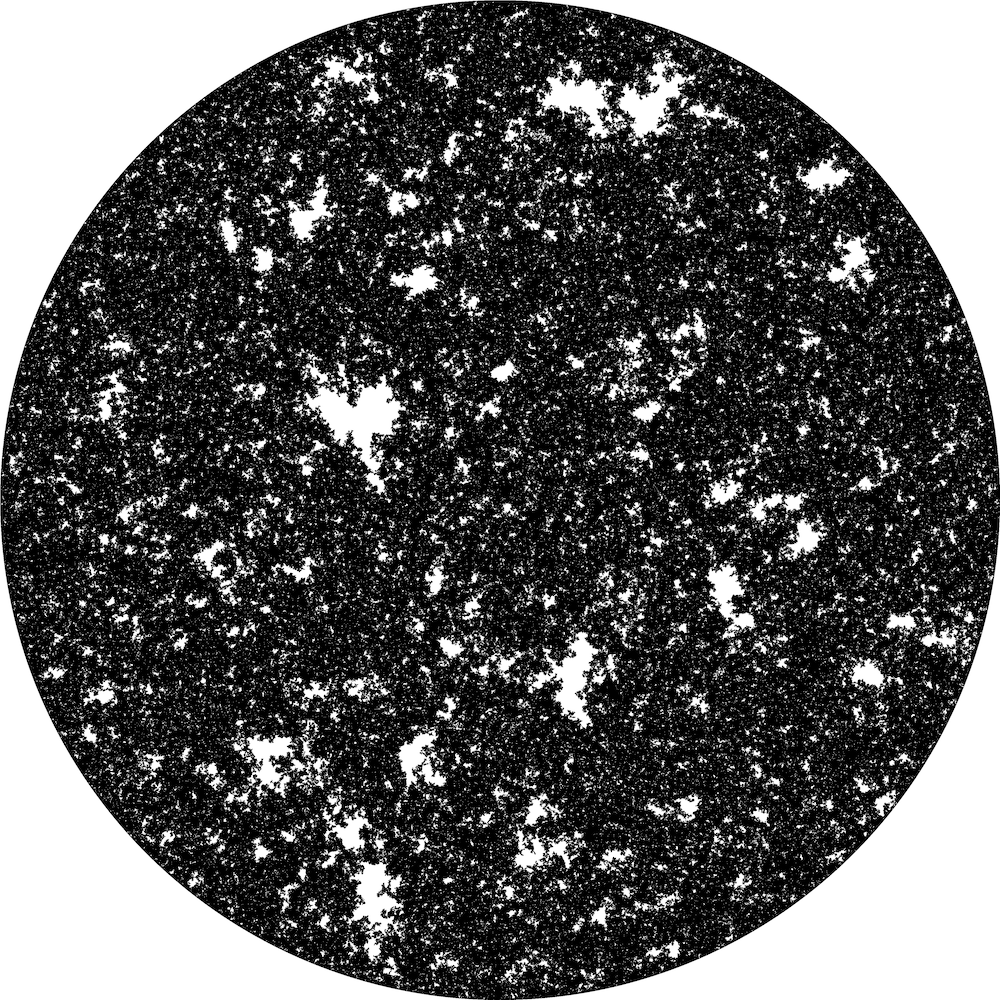} \hspace{0.01\textwidth}\includegraphics[width=0.32\textwidth]{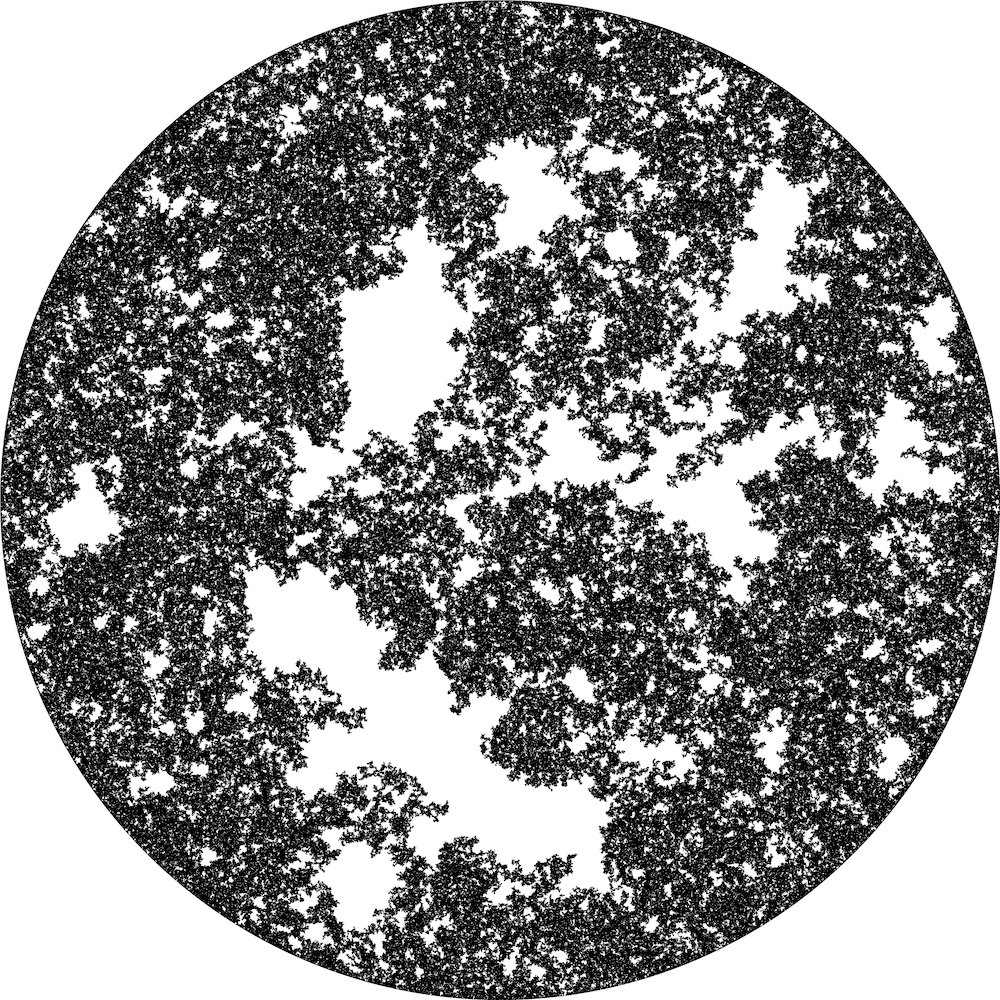}\hspace{0.01\textwidth}\includegraphics[width=0.32\textwidth]{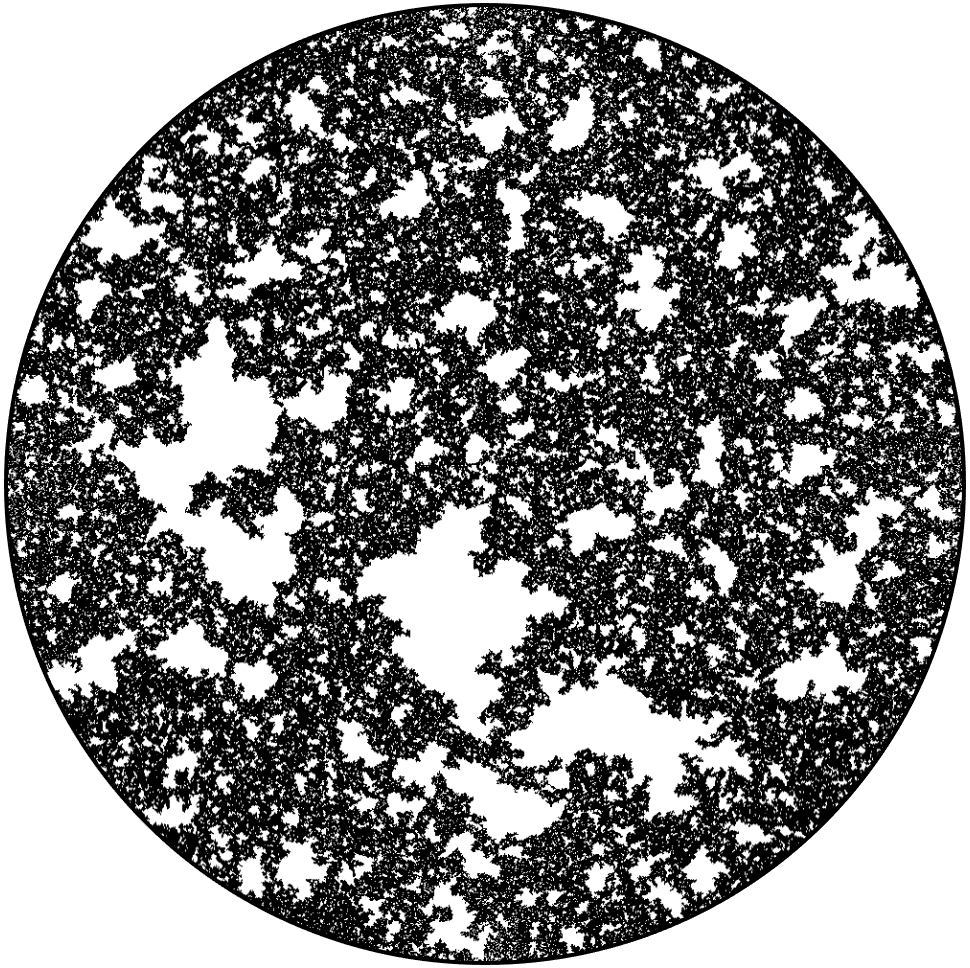}
		\caption{{\bf Left:} $\CLE_3$.  {\bf Middle:} $\CLE_4$. {\bf Right:} $\CLE_6$. The left (resp.\ middle) figure was simulated by sampling a random walk loop-soup with central charge $\tfrac{1}{2}$ (resp.\ $1$) and taking the outermost clusters. It was proved in~\cite{lupu2019convergence} that these outer boundaries converge to the $\CLE_\kappa$ for $\kappa = 3$ (resp.\ $\kappa = 4$). The right figure was simulated by sampling a critical bond percolation in the square lattice and then conformally mapping it to $\D$. This is conjectured to converge to $\CLE_6$ (the corresponding statement was proved for percolation on the triangular lattice, see~\cite{smirnov2001conformal,cn2006fullpercolation}).}
		\label{fig:CLE_4}
\end{figure}

Fix $\kappa \in (8/3,8)$.  For each simply connected domain $D \subseteq \C$, let $\p_D^{\CLE_\kappa}$ denote the law of a~$\CLE_\kappa$ in~$D$.  Then $\p_D^{\CLE_\kappa}$ satisfies the following properties.
\begin{itemize}
    \item \textbf{Conformal invariance:} If $D$ and $D'$ are simply connected domains and $\varphi: D \rightarrow D'$, then the pushforward of $\p_D^{\CLE_\kappa}$ under $\varphi$ is $\p_{D'}^{\CLE_\kappa}$.
    \item \textbf{Conformal restriction:} If $\Gamma \sim \p_D^{\CLE_\kappa}$, $\widetilde{D} \subseteq D$ is simply connected and $\widetilde{D}^\Gamma$ is the random set obtained by removing from $\widetilde{D}$ the loops in $\Gamma$ and their interiors (the points enclosed by the loops) that do not stay in $\widetilde{D}$, then the law of the loops that stay in $\widetilde{D}$, given $\widetilde{D}^\Gamma$, is the product of the laws $\p_{\widetilde{D}_j^\Gamma}^{\CLE_\kappa}$ where $\widetilde{D}_j^\Gamma$ are the connected components of $\widetilde{D}^\Gamma$.
    \item \textbf{Local finiteness:} We have $\p_{\D}^{\CLE_\kappa}$-a.s.\ that for each $\epsilon > 0$ there exists only finitely many loops of diameter greater than $\epsilon$.
\end{itemize}
In the case that $\kappa \in (8/3,4]$, $\p_D^{\CLE_\kappa}$ is characterized by these properties \cite{sw2012cle} (it is expected that there is a similar characterization in the case $\kappa \in (4,8)$ but this has not been proved; see the open question at the end of \cite{sw2012cle}).  That $\CLE_\kappa$ has these properties for $\kappa \in (8/3,4]$ was first proved in \cite{sw2012cle} using the loop-soup construction.  These properties were later proved in this case as a consequence of the iterated boundary conformal loop ensemble construction introduced in \cite{msw2017clepercolation}, the reversibility results established in \cite{ms2016ig2,ms2016ig3}, and the continuity of space-filling $\SLE_\kappa$ \cite{ms2017ig4}.  In the case that $\kappa \in (4,8)$, the conformal invariance of $\CLE_\kappa$ was proved in \cite{sheffield2009exploration} conditionally on the reversibility of $\SLE_\kappa$ for $\kappa \in (4,8)$ which was later proved in \cite{ms2016ig3}.  Finally, The conformal restriction property for $\CLE_\kappa$ in this case was carefully explained in \cite{gmq2021inversion} the local finiteness was proved in \cite{ms2017ig4} as a consequence of the continuity of space-filling $\SLE_\kappa$.

As mentioned in Section~\ref{sec:intro}, the set of points not surrounded by any loop is a random fractal, called the \emph{carpet} if $\kappa \in (8/3,4]$ or \emph{gasket} if $\kappa \in (4,8)$, and has almost sure Hausdorff dimension given by~\eqref{eq:CLEdim}.  In the case that $\kappa \in (8/3,4]$, $\CLE_\kappa$ admits two constructions.  The first is via the Brownian loop-soup and the second is using the so-called $\CLE_\kappa$ exploration tree \cite{sheffield2009exploration}.  The equivalence of these constructions was proved in \cite{sw2012cle}.  In the case that $\kappa \in (4,8)$, the only known construction of $\CLE_\kappa$ is via the exploration tree.  In what follows, we will review the loop-soup construction.

\subsubsection{Loop-soup construction}
\label{sec:loop-soup}

Let $\mu^\sharp(z,w;t)$ denote the law of the Brownian bridge in $\C$ from $z$ to $w$ in time $t$. Then the (unrooted) Brownian loop measure in $\C$ is defined as
\begin{align*}
    \mu_{\C}^{\textup{loop}} = \int_{\C} \int_0^\infty \frac{1}{2\pi t^2} \mu^\sharp(z,z;t) dt dz.
\end{align*}
The Brownian loop measure in $D \subseteq \C$, denoted $\mu_D^{\textup{loop}}$, is defined as the restriction of $\mu_{\C}^{\textup{loop}}$ to the loops that are contained in $D$. The Brownian loop-soup in $D\subseteq \C$ with intensity $c > 0$ is defined as a Poisson point process with intensity given by $c\mu_D^{\textup{loop}}$. The Brownian loop-soup is conformally invariant, that is, the conformal image of a Brownian loop-soup is a Brownian loop-soup in the new domain, with the same intensity. For more on the Brownian loop-soup, see \cite{lw2004loopsoup}.

Let $\wc{\Gamma} = (\CL_j)_{j\in J}$ be a Brownian loop-soup with intensity $c \in (0,1]$. Let $\Gamma$ be the collection of outer boundaries of the clusters of $\wc{\Gamma}$. By \cite{sw2012cle}, $\Gamma$ is a $\CLE_\kappa$ where $\kappa \in (8/3,4]$ is determined by the relation
\begin{align*}
    c = \frac{(3\kappa-8)(6-\kappa)}{2\kappa}.
\end{align*}
This construction gives a natural coupling of $\CLE_\kappa$'s with $\kappa \in (8/3,4]$ induced by the Poisson point process. More precisely, we can couple loop-soups $(\wc{\Gamma}_c)_{c\in(0,1]}$ such that $\wc{\Gamma}_{c_1} \subseteq \wc{\Gamma}_{c_2}$ whenever $c_1 < c_2$. This in turn gives a monotone coupling of $\CLE_\kappa$ which we will use in our construction of the natural measure on $\CLE_4$. This is useful, as the natural LQG measure on $\CLE$, which we will use in our construction, has not previously been defined for $\kappa = 4$.

\subsection{Gaussian free field}

We now introduce the Gaussian free field (GFF). For more details, see \cite{sheffield2007gff,sheffield2016zipper}. Let $D$ be a Jordan domain and denote by  $C_0^\infty(D)$ the smooth functions on $D$ with compact support. We denote by $H_0(D)$ the Hilbert space closure of $C_0^\infty(D)$ with respect to the Dirichlet inner product
\begin{align*}
    (f,g)_\nabla = \frac{1}{2\pi} \int_D \nabla f(z) \cdot \nabla g(z) dz.
\end{align*}

The \emph{zero-boundary GFF} $h$ on $D$ is given by
\begin{align*}
    h = \sum_{n=1}^\infty \alpha_n \phi_n
\end{align*}
where $(\phi_n)_{n=1}^\infty$ is a $(\cdot,\cdot)_\nabla$-orthonormal basis of $H_0(D)$ and $(\alpha_n)_{n=1}^\infty$ is an i.i.d.\ sequence of $N(0,1)$ random variables. Clearly, $h$ is not a function, however, the partial sums converge in the Sobolev space $H^{-\epsilon}(D)$ for every $\epsilon > 0$. The law of $h$ is independent of the choice of orthonormal basis. It follows from the conformal invariance of the Dirichlet inner product that the law of $h$ is conformally invariant as well. That is, if $\varphi: D \rightarrow \widetilde{D}$ is a conformal map, then
\begin{align*}
    \wt{h} \coloneqq h \circ \varphi^{-1} = \sum_{n=1}^\infty \alpha_n \phi_n \circ \varphi^{-1}
\end{align*}
is a zero-boundary GFF in $\widetilde{D}$.

Let $U \subseteq D$ be open. Integrating by parts, it is easy to see that $H_0(D)$ admits the $(\cdot,\cdot)_\nabla$-decomposition $H_0(D) = H_0(U) \oplus H_0^\perp(U)$ where $H_0^\perp(U) = H_{0,D}^\perp(U)$ is the set of functions in $H_0(D)$ which are harmonic on $U$. Hence, we get a decomposition of $h$ into $h = h_U + h_U^\perp$, with $h_U$ and $h_U^\perp$ independent, where $h_U$ is a zero-boundary GFF in $U$ and $h_U^\perp$ is a distribution which agrees with $h$ on $D \setminus U$ and is harmonic on $U$. In particular, $h_U^\perp$ can be thought of as the harmonic extension to $U$ of the values of $h$ on $\partial U$. This is the Markov property of the GFF.

A GFF with more general boundary data, say $f$, is defined as $h+F$ where $h$ is a zero-boundary GFF and $F$ is the harmonic extension of the values of $f$ to the domain.

Equivalently, one can define the zero-boundary GFF $h$ as the centered Gaussian process, $h: H_0(D) \rightarrow L^2(\p)$, with the Green's function as its correlation kernel. That is, $(h,f), f \in H_0(D)$ is a collection of zero mean Gaussian random variables with correlations given by
\begin{align*}
	\E[(h,f)(h,g)] = \int_{D \times D} f(z) G(z,w) g(w) dz dw,
\end{align*}
where $G=G_D$ is the Green's function for $D$ with Dirichlet boundary data. The conformal invariance of $h$ follows from the conformal invariance of Green's functions.

We will next give the definition of the free boundary GFF. Let $H(D)$ be the Hilbert space closure with respect to $(\cdot,\cdot)_\nabla$ of the space of functions $f \in C^\infty(D)$ such that $\int_D f dz = 0$. Then the \emph{free boundary GFF} is defined by the sum
\begin{align*}
    h = \sum_{n=1}^\infty \alpha_n \psi_n,
\end{align*}
where $(\psi_n)_{n=1}^\infty$ is a $(\cdot,\cdot)_\nabla$-orthonormal basis of $H(D)$ and $(\alpha_n)_{n=1}^\infty$ is an i.i.d.\ sequence of $N(0,1)$ random variables. Just as in the zero-boundary case, the free boundary GFF is conformally invariant. We note that as the free boundary GFF is only defined on test functions of zero mean, it is actually not canonically defined in a space of distributions, but rather in a space of distributions modulo additive constant.

We may decompose a free boundary GFF $h$ on $D$ as the sum of a zero-boundary GFF on $D$ and the harmonic extension of the values of $h$ on $\partial D$ to $D$. We note that this harmonic extension will be a random harmonic function. Consequently, the zero-boundary GFF and the free boundary GFF are mutually absolutely continuous away from the boundary.

\begin{rmk}[Radial/lateral decomposition]\label{rmk:radlat}
Let $H^R(\UH)$ and $H^L(\UH)$ be the subspaces of functions in $H(\UH)$ which are constant on and have mean zero on each semi-circle centered at zero, respectively. Then $H(\UH) = H^R(\UH) \oplus H^L(\UH)$. The projection of a free boundary GFF $h$ onto $H^R(\UH)$ is the function $h^R=h_{|\cdot|}(0)$ which on each semi-circle centered at $0$ takes the average value of the field $h$ on that semi-circle. The projection of $h$ onto $H^L(\UH)$ is given by $h^L = h-h_{|\cdot|}(0)$. We call $h^R$ and $h^L$ the radial and lateral parts of $h$, respectively. Clearly, $h^R$ and $h^L$ are independent. Note that while~$h^R$ is only defined modulo additive constant, $h^L$ is actually well-defined.  There is also an analogous radial/lateral decomposition for $H(\SSS)$ as the orthogonal sum of the spaces $H^R(\SSS)$, $H^L(\SSS)$ where the former consists of those functions in $H(\SSS)$ which are constant on lines of the form $t + (0,i\pi)$ and the latter consists of those functions in $H(\SSS)$ which have mean zero on such lines.
\end{rmk}

Finally, we mention the notion of a local set. In practice, this is just a random set for which a domain Markov property holds. More precisely, if $h$ is a GFF on $D$, then we say that random closed subset $A$ of $\overline{D}$ is a local set for $h$ if there is a coupling of $A$, $h$ and a field $h_A$ such that
\begin{itemize}
    \item $h_A$ can be represented as a harmonic function on $D \setminus A$,
    \item Conditionally on $(h,A)$, $h^A \coloneqq h - h_A$ is a GFF in $D \setminus A$ with zero boundary conditions.
\end{itemize}
In fact, $h_A$ can be seen to be a deterministic function of $h$ and $A$, so the coupling concerns only $h$ and $A$.

\subsection{Liouville quantum gravity}
\label{sec:LQG}
Fix $\gamma \in (0,2)$.  Informally, a Liouville quantum gravity (LQG) surface is a random two-dimensional Riemannian manifold parameterized by a domain $D$ with metric tensor given by
\begin{align*}
    e^{\gamma h(z)} dx \otimes dy
\end{align*}
where $h$ is some form of the GFF on $D$.  Of course, this does not make literal sense, since the exponential of a distribution is not well-defined. Instead, one has to define LQG surfaces via renormalization.

For $z \in D$ and $0 < \epsilon < \dist(z,\partial D)$, we let $h_\epsilon(z)$ denote the average value of $h$ on $\partial B(z,\epsilon)$. Then the $\gamma$-LQG area measure on $D$ is a random measure which is defined as the weak limit
\begin{align*}
    \mu_h(dz) = \lim_{\epsilon \rightarrow 0} \epsilon^{\gamma^2/2} e^{\gamma h_\epsilon(z)} dz
\end{align*}
where $dz$ denotes Lebesgue measure.

If $h$ is a free boundary GFF (or any Gaussian field which is locally absolutely continuous with respect to a free boundary GFF plus a continuous function), then it is also possible to define a boundary length measure. For $x \in \partial D$, let $h_\epsilon(x)$ be the average value of $h$ on $\partial B(x,\epsilon) \cap D$. Then, on a linear segment of $\partial D$, we define the quantum length measure as
\begin{align*}
    \nu_h(dx) = \lim_{\epsilon \rightarrow 0} \epsilon^{\gamma^2/4} e^{\frac{\gamma}{2} h_\epsilon(x)} dx,
\end{align*}
where $dx$ is Lebesgue measure on $\partial D$.

It is also possible to construct a metric for LQG \cite{ms2020lqgtbm1,ms2021lqgtbm2,gm2021metric,dddf2020tightness}, though we will not need this in the present paper.

Let $\widetilde{D}$ be a simply connected domain, $\psi: \widetilde{D} \rightarrow D$ a conformal map and write
\begin{align*}
    \wt{h} = h \circ \psi + Q\log|\psi'|, \quad Q = \frac{2}{\gamma} + \frac{\gamma}{2}.
\end{align*}
Then $\mu_h$ is the pushforward of the measure $\mu_{\wt{h}}$ under $\psi$. That is, for each $A \subseteq \widetilde{D}$, $\mu_{\wt{h}}(A) = \mu_h(\psi(A))$ holds a.s. Similarly, $\nu_h$ is the pushforward of $\nu_{\wt{h}}$ under $\psi$. Thus, the random surface does not depend on which domain $D$ we parameterize it by and hence a quantum surface is defined to be an equivalence class of pairs $(D,h)$ consisting of a simply connected domain and a distribution under the relation transformation
\begin{align}
\label{eq:QS}
    (D,h) \mapsto (\psi^{-1}(D), h \circ \psi + Q\log|\psi'|).
\end{align}
Note that this provides a definition of $\nu_h$ even for boundary arcs which are not linear: one just maps the domain to $\UH$ (for example) and measures the length there.

The measures $\nu_h$ can actually measure the length of some curves inside of the domain as well. This is done by mapping the curve to the boundary and measuring the length of the boundary segment that is the image of the curve. That is, if $\eta$ is a curve in $\UH$ and $g: \UH \setminus \eta \rightarrow \UH$ is a conformal map which extends continuously to $\eta$ (in the sense of prime ends), then the quantum length of the left (resp. right) side of $\eta$, $\eta_L$ (resp. $\eta_R$), is defined as the measure of $g(\eta_L)$ (resp. $g(\eta_R)$) with respect to $\nu_{\bar{h}}$, where 
\begin{align*}
    \bar{h} = h \circ g + Q\log|g'|.
\end{align*}
In \cite{sheffield2016zipper}, it is proven that if $\kappa \in (0,4)$, then the length of the left and right sides of an independent $\SLE_\kappa$ curve drawn on a quantum surface called a $(\gamma - 2\gamma^{-1})$-quantum wedge (corresponding to a weight of $4$), $\gamma = \sqrt{\kappa}$, can be made sense of and that they are equal at any given time. In \cite{dms2014mating}, the authors prove that the quantum length of the left and right sides of an $\SLE_{\kappa'}$ hull, $\kappa' \in (4,8)$, drawn on an independent $(\tfrac{4}{\gamma} - \tfrac{\gamma}{2})$-quantum wedge (corresponding to a weight of $3\gamma^2/2-2$), $\gamma = 4/\sqrt{\kappa'}$, are given by independent $\kappa'/4$-stable L\'{e}vy processes with only negative jumps provided the curve has the appropriate time parameterization.  The negative jumps correspond to the components that the curve disconnects from $\infty$ and the size of a downward jump gives the quantum length of the component.  We will refer to this time parameterization as \emph{generalized quantum length}.

Since the time parameterization of a stable L\'evy process can be recovered from its ordered collection of jumps, we can recover the generalized quantum length from the ordered collection of quantum lengths of the components disconnected from $\infty$ by the curve.  More concretely, let $N_\epsilon^{\eta,h}(dz)$ denote the counting measure on bubbles cut out by $\eta$ of LQG boundary length in $[\epsilon,2\epsilon]$, with respect to $h$.  The boundary of each such bubble is absolutely continuous with respect to an $\SLE_\kappa$ for $\kappa = 16/\kappa' \in (0,4)$ and therefore it has a well-defined quantum length. The generalized quantum length of $\eta$ is the random measure defined by
\begin{align*}
    \sigma^{\eta,h}(dz) \coloneqq \lim_{\epsilon \rightarrow 0} \epsilon^{\wh{\alpha}} N_\epsilon^{\eta,h}(dz),
\end{align*}
where $\wh{\alpha} = \kappa'/4 = 4/\gamma^2$. This is a volume measure such that $\sigma^{\eta,h}(\eta([0,t]))$ is continuous and strictly increasing in $t$. Moreover, we write $\sigma_I^{\eta,h}$ and $N_\epsilon^{\eta,h,I}$ for these quantities, restricted to the time interval $I$.

Note that while we require a distribution which looks locally like a free boundary GFF to measure the length of segments of $\partial D$, we can actually measure the length of a curve in $D$ with a zero-boundary GFF as well, since inside of $D$, the law of a zero-boundary GFF is absolutely continuous with respect to that of a free boundary GFF.

\begin{rmk}\label{rmk:qlengthCLE}
We note that quantum length and generalized quantum length of $\SLE$ processes provide us with a notion of length of $\CLE$ loops. For a single loop $\CL$ of a $\CLE_{\kappa'}$, we use the same notation as above, that is, $\sigma^{\CL,h}$ denotes its quantum length with respect to the field $h$, and $N_\epsilon^{\CL,h}$ denotes the counting measure on the smaller loops with sizes in $[\epsilon,2\epsilon]$ that $\CL$ traces (both on the left and on the right).
\end{rmk}

\subsubsection{Weight $3\gamma^2/2-2$ quantum wedge}

In what follows in Section~\ref{sec:npSLE}, it will be important to have the exact definition of a weight $3\gamma^2/2-2$ quantum wedge (i.e., a $(4/\gamma-\gamma/2)$-quantum wedge) so we will review it here.  The definition is  easiest to give when the surface is parameterized by the infinite strip $\SSS = \R \times (0,\pi)$ rather than by $\UH$.  One takes the projection $h$ of the field onto $H^L(\SSS)$ to be given by the corresponding projection of a GFF on $\SSS$ with free boundary conditions.  The projection of the field onto $H^R(\SSS)$ is taken independently to be the function whose common value on $\{t\} \times (0,\pi)$ is given by $Y_t = B_{2t} + (\gamma - 2/\gamma)t$ where $B$ is a two-sided Brownian motion with $B_0 = 0$ and conditioned so that $Y_t > 0$ for all $t > 0$.  Since a quantum surface is defined modulo~\eqref{eq:QS} this in fact defines a particular embedding of a weight $3\gamma^2/2-2$ quantum wedge into $\SSS$ which is sometimes called the circle average embedding.  The circle average embedding of a weight $3\gamma^2/2-2$ wedge parameterized by $\UH$ is defined by starting with the circle average embedding of the surface parameterized by $\SSS$ and then mapping to $\UH$ using the map $z \mapsto e^z$ and applying~\eqref{eq:QS}.

\subsubsection{Ordinary and generalized quantum disks}

Two types of quantum surfaces that are of particular importance are the quantum disk and the generalized quantum disk. The exact definitions are not important for understanding the results, however, we will provide the definition of the quantum disk and mention briefly how to construct the generalized quantum disk.

While we will often let it be parameterized by the unit disk, it is convenient to define it parameterized by the infinite strip $\SSS$, since it then takes a simple form. We can then parameterize it by $\D$ using a conformal map and~\eqref{eq:QS}.

Fix $\gamma \in (0,2)$. The \emph{infinite measure $\mathcal{M}^{\textup{disk}}$ on quantum disks} is a measure on doubly marked quantum surfaces $(\SSS,h^D,-\infty,\infty)$. A quantum disk $(\SSS,h^D,-\infty,\infty)$ can be sampled from $\mathcal{M}^{\textup{disk}}$ as follows. Sample a process $X$ from the infinite excursion measure of a Bessel process of dimension $3-\frac{4}{\gamma^2}$ (described in \cite[Remark~3.7]{dms2014mating}) and let $Z$ be the process defined by taking $\frac{2}{\gamma} \log X$ and reparameterizing it to have quadratic variation $2dt$. Then we define $h^D$ as the distribution with average value $Z_t$ on each vertical line segment $\{t\} \times (0,\pi)$ and lateral part $h^D - Z_{\textup{Re}(\cdot)}$ given by that of an independent free boundary GFF on $\SSS$ (recall Remark~\ref{rmk:radlat}).

For $\ell>0$ the quantum disk with boundary length $\ell$ is is the law on surfaces given by $\mathcal{M}^\textup{disk}$ conditioned on $\nu_{h^D}(\partial \mathcal{S}) = \ell$ (see \cite[Section~4.5]{dms2014mating}).

A generalized quantum disk is constructed by considering a $\frac{4}{\gamma^2}$-stable looptree (see e.g. \cite{ck2014looptrees}) and assigning to each of the loops the conformal structure of an independent $\gamma$-quantum disk with boundary length given by the length of the loop and marked point given by the point on the loop which is closest to the root.

Just as in the case of the free boundary GFF, it holds that inside of the domain, the law the quantum disk -- and hence the generalized quantum disk -- is absolutely continuous with respect to the law of a zero-boundary GFF.

\subsubsection{Natural LQG measure on the $\CLE$ carpet/gasket}\label{sec:nmCLE}
Just as for $\SLE$, there is a natural LQG measure on the $\CLE$ carpet. Fix parameters
\begin{align*}
    \kappa \in (8/3,4), \quad \gamma = \sqrt{\kappa}, \quad \alpha = 4/\kappa,
\end{align*}
and consider a $\CLE_\kappa$ process $\Gamma$ drawn on an independent $\gamma$-quantum disk $h$. For an open set $O$ let $N_\epsilon^{\Gamma,h}(O)$ denote the number of loops which are contained in $O$ and have quantum length in $[\epsilon,2\epsilon]$ with respect to $h$. Then the limit
\begin{align}\label{eq:CLEnm}
    \Lambda^{\Gamma,h}(O) = \lim_{\epsilon \rightarrow 0} \epsilon^{\alpha+\frac{1}{2}} N_\epsilon^{\Gamma,h}(O)
\end{align}
exists in probability for each $O$ and defines a random measure which we interpret as the natural LQG measure on the $\CLE_\kappa$ carpet; see \cite[Theorem~1.3]{msw2020simpleclelqg}.

The case of non-simple $\CLE$ is similar. Suppose $\kappa' \in (4,8)$, $\gamma = 4/\sqrt{\kappa'}$ and $\alpha' = 4/\kappa'$, and $\Gamma$ is a $\CLE_{\kappa'}$ drawn on an independent generalized quantum disk $h$. For any open sen $O$, let $N_\epsilon^{\Gamma,h}(O)$ denote the number of loops of $\Gamma$ which are contained in $O$ and have generalized quantum length in $[\epsilon,2\epsilon]$ with respect to $h$ (recall here that the loops are $\SLE_{\kappa'}$-type loops and hence their lengths are measured accordingly). Then, the natural LQG measure on the $\CLE$ gasket is given by the limit~\eqref{eq:CLEnm} with this $N_\epsilon^{\Gamma,h}(O)$ and $\alpha'$ in place of $\alpha$ (see \cite{msw2020nonsimpleclelqg}).

\section{Natural parameterization of self-intersecting SLE}\label{sec:npSLE}

Fix $\kappa' \in (4,8)$ and let $\gamma = 4/\sqrt{\kappa'}$ and $\wh{\alpha} = \kappa'/4$. Using LQG measures, we want to construct a $(1+\frac{\kappa'}{8})$-dimensional volume measure on $\eta \sim \SLE_{\kappa'}$ which will be (a constant multiple of) the natural parameterization of $\eta$. We stress that by Theorem~\ref{thm:SLE_unique}, we need only show that the constructed measures $(\mu^t)_{t \geq 0}$ satisfy the correct conformal covariance~\eqref{eq:confcovSLE}, have the same law (under proper conditioning) and are locally finite. The results of~\cite{benoist2018natural} lead us to considering a measure of the form
\begin{align*}
    \mu^0(dz) = F(z) \E\!\left[\sigma^{\eta^0,h^0}(dz) \, \middle| \, \eta^0 \right]
\end{align*}
where $\eta^0 = \eta$, $h^0$ is a zero-boundary GFF on $\UH$ independent of $\eta^0$, and $F(z) = r_{\UH}(z)^{-2/\gamma^2}$ where $r_{\UH}(z)$ is the conformal radius of $\UH$ as seen from $z$.  We recall that $(f_t)$ is the centered Loewner flow associated with $\eta$, $\eta^t(u) = f_t(\eta(t+u))$ for each $u \geq 0$, $\psi_t = f_t^{-1}$, and $\mu^t = |\psi_t'(z)|^{-d} (\mu \circ \psi_t)(dz)$.  We define the volume measure $\wt{\mu}^s$ on $\eta^s$ (recall~\eqref{eq:unzipped_curve}) by
\begin{align*}
    \wt{\mu}^s(dz) = F(z) \E\!\left[\sigma^{\eta^s,\wt{h}^s}(dz) \, \middle| \, \eta^s \right]
\end{align*}
where $\wt{h}^s$ is a zero-boundary GFF in $\UH$ independent of $\eta^0$.  It is clear that $(\eta^s,\wt{\mu}^s)$ and $(\eta^0,\mu^0)$ have the same law. The goal of this section is to prove that $\wt{\mu}^s = \mu^s$ (recall~\ref{eq:confcovSLE}) hence deduce that $(\eta^s,\mu^s)$ and $(\eta^0,\mu^0)$ have the same law. In order to do this, we begin by showing that $\mu^0$ is a.s.\ locally finite as a measure on $\UH$. To be explicit, we say that $\mu^0$ is an a.s.\ locally finite as a measure on $\UH$ if we have that a.s.\ $\mu^0(K) < \infty$ for every $K \subseteq \UH$ compact.
Note that this is different from the local finiteness condition of Theorem~\ref{thm:SLE_unique}, but we shall prove that it implies the latter condition, which requires that for all $0<s<t$, we have that $\mu^0(\eta^0([s,t]))$ is a.s.\ finite.

\subsection{Local finiteness of $\mu^0$ as a measure on $\UH$}

The goal of this subsection is to prove the following.
\begin{lem}
\label{lem:locally_finite_SLE}
Almost surely, $\mu^0$ is locally finite as a measure on $\UH$.
\end{lem}

This is proved by first proving that if $h^w$ is a weight $3\gamma^2/2-2$ quantum wedge with the circle average embedding, taken to be independent of $\eta^0$, and $h^w = h^0 + \Fh$ is its decomposition into a zero-boundary GFF and a harmonic function, then $\E[\sigma^{\eta^0,h^w}| \Fh](\D \cap \UH) < \infty$ a.s. The result then follows from a comparison between $h^w$ and $h^0$ in $\D \cap \UH$ and a scaling argument. We begin by proving some moment bounds for $\mu_{h^w}$.

\begin{lem}
\label{lem:rectangle_bound}
Fix $\gamma \in (\sqrt{2},2)$ and let $h^w$ be a weight $\tfrac{3\gamma^2}{2} -2$ quantum wedge parameterized by $\SSS$.  We take the embedding so that the projection of $h^w$ onto $H^R(\SSS)$ last hits $0$ on the line $\{0\} \times (0,\pi)$.  For each $k \in \Z$ we let $A_k = [k,k+1]\times(0,\pi)$ and let $h^w = h^0 + \Fh$ be the decomposition of $h^w$ into a zero-boundary GFF $h^0$ and its harmonic part $\Fh$ where $h^0$, $\Fh$ are independent. Then for small enough $p>0$ we have that
\begin{align*}
	\E\!\left[ \E\!\left[ \mu_{h^w}(A_0) \giv \Fh \right]^p \right] < \infty.
\end{align*}
Moreover for $p,\epsilon>0$ small enough and $\wt{p} = p+\epsilon$ there exists a constant $C = C_{p,\epsilon} > 0$ so that for $k \geq 1$ we have that
\begin{align*}
	\E\!\left[ \E[ \mu_{h^w}(A_k) \giv \Fh]^p \right] \leq C e^{k( \wt{p}(\gamma^2-2) + \wt{p}^2 \gamma^2)}.
\end{align*}
Finally, for small enough $p,\epsilon > 0$ and $\wt{p} = p+\epsilon$ there exists a constant $C = C_{p,\epsilon} > 0$ such that for $k \leq 0$ we have that
\begin{align*}
	\E\!\left[ \E[\mu_{h^w}(A_{k-1}) \giv \Fh]^p \right] \leq C e^{-|k| \wt{p}(\gamma^2-2) + \wt{p}^2 \gamma^2 |k|}.
\end{align*}
\end{lem}
\begin{proof}
By the independence of $h^0$ and $\Fh$ we have from \cite[Proposition~1.2]{ds2011kpz} that
\begin{align*}
\E\!\left[ \mu_{h^w}(A_0) \, \middle| \, \Fh \right] = \int_{A_0} e^{\gamma \Fh(z)} r_{\SSS}(z)^{\gamma^2/2} dz
\end{align*}
where $r_\SSS(z)$ is the conformal radius of $\SSS$ as seen from $z$.

Let $\wt{\CQ}$ be a Whitney square decomposition of $A_0$ (see for example~\cite[Section~I.4]{gm2005harmonic}).  Fix $\wh{\delta} > 0$ and let $\CQ$ be formed by subdividing each $Q \in \wt{\CQ}$ into a finite number of subsquares (depending only on $\wh{\delta}$) of equal size so that if $Q \in \CQ$ then $\diam(Q) / \dist(Q,\partial \SSS) \leq \wh{\delta}$.  Fix some $p \in (0,1)$ to be determined later. By the concavity of the function $x \mapsto x^p$ for $p \in (0,1)$ we have that
\begin{align*}
\left(  \int_{A_0} e^{\gamma \Fh(z)} r_{\SSS}(z)^{\gamma^2/2} dz \right)^p \lesssim \sum_{Q \in \CQ} e^{\gamma p \sup_{z \in Q} |\Fh(z)|} \dist( Q, \partial \SSS)^{p \gamma^2/2} | Q |^p,
\end{align*}
where $|Q|$ is the area of $Q$. Taking expectations, we have that
\begin{align}\label{eq:bound_rectangle_measure}
	\E\!\left[ \E\!\left[ \mu_{h^w}(A_0) \, \middle| \, \Fh \right]^p \right] \lesssim \sum_{Q \in \CQ} \E\!\left[e^{\gamma p \sup_{z \in Q} |\Fh(z)|} \right] \dist( Q, \partial \SSS)^{p \gamma^2/2} | Q |^p.
\end{align}
Next, note that $h^w$ differs from a free boundary GFF only in the vertical direction. More precisely, (recall Remark~\ref{rmk:radlat}), the projection of $h^w$ onto $H^L(\SSS)$ has the same law as the projection of a free boundary GFF $h^f$ on $\SSS$ onto $H^L(\SSS)$. Furthermore, let $Y_t$ denote the average value of $h^w$ on $\{t\} \times (0,\pi)$, that is, its projection on $H^R(\SSS)$.  Then the law of $Y_t$ is given by that of $B_{2t} + (\gamma - \tfrac{2}{\gamma})t$ where $B$ is a two-sided Brownian motion with $B_0 = 0$ and conditioned so that $B_{2t} + (\gamma - \tfrac{2}{\gamma})t > 0$ for all $t > 0$. If $U_t$ denotes the projection of $h^f$ onto $H^R(\SSS)$ then the law of $U_t$ is equal to that of $\wt{B}_{2t}$ where $\wt{B}$ is a two-sided Brownian motion with $\wt{B}_0 = 0$.  Consequently, it follows that if~$\Fh^f$ denotes the harmonic part of $h^f$ then for each $Q \in \CQ$,
\begin{align}\label{eq:harmonic_parts_comparable}
	\E\!\left[e^{\gamma p \sup_{z \in Q} |\Fh(z)|} \right] \asymp \E\!\left[e^{\gamma p \sup_{z \in Q} |\Fh^f(z)|} \right]
\end{align}
where the implicit constant is independent of $Q$. For each $k$ we let $\CQ_k = \{Q \in \CQ: 2^{-k} \leq \dist(Q,\partial\SSS) < 2^{-k+1}\}$.  Fix $\delta > 0$.  By decreasing the value of $\wh{\delta} > 0$ if necessary \cite[Lemma~A.2]{kms2021regularity} together with~\eqref{eq:harmonic_parts_comparable} implies for $Q \in \CQ_k$ that
\begin{align}\label{eq:harmonic_part_bound}
	\E\!\left[e^{\gamma p \sup_{z \in Q} |\Fh(z)|} \right] = O(2^{k(1+\delta) \gamma^2 p^2}),
\end{align}
where the implicit constant depends only on $\delta$. 

Next, note that $|\CQ_k| \asymp2^k$ with the implicit constant depending only on $\delta$. Moreover, since $\dist(Q,\partial \SSS) \asymp \diam(Q)$ (with constants depending only on $\delta$), we have for all $Q \in \CQ_k$ that
\begin{align*}
	\dist(Q,\partial\SSS)^{p\gamma^2/2} |Q|^p \asymp 2^{-kp(2 + \gamma^2/2)}
\end{align*}
where the implicit constants depend only on $\delta$. Combining this with~\eqref{eq:bound_rectangle_measure} and~\eqref{eq:harmonic_part_bound} it follows that
\begin{align}\label{eq:sum_bound}
	\E\!\left[ \E\!\left[ \mu_{h^w}(A_0) \, \middle| \, \Fh \right]^p \right] \lesssim \sum_{k=0}^\infty O\!\left(2^{k((1+\delta) \gamma^2 p^2 -p(2+\gamma^2/2)+1)}\right).
\end{align}
Thus if $(1+\delta) \gamma^2 p^2 -p(2+\gamma^2/2)+1 < 0$ we have that the sum converges and the result follows. Note that the polynomial $q(p) = \gamma^2 p^2 -p(2+\gamma^2/2)+1$ attains its minimum $-\tfrac{(4-\gamma^2)^2}{16\gamma^2} < 0$ at the value $p = \tfrac{4+\gamma^2}{4\gamma^2}$.  Consequently, for $\delta > 0$ very small there exists $p_0 > 0$ so that the polynomial $q_\delta(p) = (1+\delta)\gamma^2p^2 - p(2+\gamma^2/2)+1$ satisfies $q_\delta(p_0) < 0$ and hence such that~\eqref{eq:sum_bound} is bounded for $p = p_0$. By Lyapunov's inequality, the result holds for all $0< p < p_0$ as well.

We now turn to the bound on $\mu_{h^w}(A_k)$ for $k \geq 1$.  By H\"{o}lder's inequality, with $p',q' > 1$ satisfying $\tfrac{1}{p'} + \tfrac{1}{q'} = 1$, we have that
\begin{align*}
	\E\!\left[ \E[\mu_{h^w}(A_k) \giv \Fh]^p \right] \leq \E\!\left[ e^{p' p \gamma Y_k} \right]^{\frac{1}{p'}} \E\!\left[ \E[\mu_{h^w - Y_k}(A_k) \giv \Fh]^{q'p} \right]^{\frac{1}{q'}}.
\end{align*}

Let $\sigma = \inf\{t \geq 0 : Y_t = 1\}$.  Then we note that $Y_{t+\sigma}$ has the same law as $Z_t = B_{2t} + (\gamma-2/\gamma) t$ conditioned on the positive probability event $F$ that $Z_t > 0$ for all $t \geq 0$ where $B$ is a standard Brownian motion with $B_0 = 1$.  Let $\F_t = \sigma(Y_s : s \leq t)$.  Thus for $b > 0$ we have that
\begin{align*}
\E[ e^{b Y_t} ]
&\leq e^b + \E[ \exp(b Y_t) \one_{t \geq \sigma}]
 = e^b + \E[ \E[ \exp(b Y_t) \giv \F_\sigma] \one_{t \geq \sigma}].
\end{align*}
We moreover have on $\{t \geq \sigma\}$ that
\begin{align*}
\E[ \exp(b Y_t) \giv \F_\sigma]
&= \E[ \exp(b Z_{t-\sigma}) \giv F]
 \leq \frac{\E[ \exp(b Z_{t-\sigma})]}{\p[F]}\\
 &= \frac{\exp( (b(\gamma-2/\gamma) + b^2)(t-\sigma))}{\p[F]}\\
 &\leq \frac{\exp( (b(\gamma-2/\gamma) + b^2)t)}{\p[F]}.
\end{align*}
Inserting this into the above gives that
\[ \E[ \exp(b Y_t)] \leq e^b + \frac{1}{\p[F]} e^{(b(\gamma-2/\gamma) + b^2)t}.\]
Applying this for $b = p p' \gamma$ and $t =k$ gives us that
\begin{align*}
	\E\!\left[ e^{p p' \gamma Y_k} \right] \lesssim e^{k((p p')^2 \gamma^2 + p p' \gamma^2 - 2p p')}.
\end{align*}
The same argument works to bound the exponential moments of $(Y_{k+t}-Y_k)_{t \in [0,1]}$. Hence,
\begin{align*}
	\E\!\left[ \E[ \mu_{h^w - Y_k}(A_k) \giv \Fh]^{q'p} \right] \lesssim \E\!\left[ \E[\mu_{h^w}(A_0) \giv \Fh]^{q'p} \right] < \infty,
\end{align*}
whenever $q'p$ is sufficiently small.  Hence, the second part of the lemma follows.

Finally, we bound the moments of the rectangles $A_{-k}$ for $k \geq 1$.  We have that
\begin{align*}
	\E\!\left[ \E[ \mu_{h^w}(A_{-k}) \giv \Fh]^p \right]
	&= \E\!\left[ e^{\gamma Y_{1-k}}\E[ \mu_{h^w - Y_{1-k}}(A_{-k}) \giv \Fh]^p \right]\\
	&\leq \E[ e^{\gamma p p' Y_{1-k}}]^{1/p'} \E[ \E[ \mu_{h^w - Y_{1-k}}(A_{-k}) \giv \Fh]^{p q'}]^{1/q'}
\end{align*}
The second term on the right-hand side is bounded by a constant times $\E[ \E[\mu_{h^w}(A_0) \giv \Fh]^{p q'}]^{1/q'}$ (since the average process of the latter is a Brownian motion with positive drift, conditioned to stay positive, instead of an average process given by a Brownian motion with positive drift), which is finite for $p > 0$ small, and the former term is
\begin{align*}
	\E[ e^{p p' \gamma Y_{1-k}} ] \lesssim e^{(1-k) p p' (\gamma^2 - 2) + (p p')^2 \gamma^2 (k-1)} \lesssim e^{-k p p' (\gamma^2 - 2) + (p p')^2 \gamma^2 k}. 
\end{align*}
where the implicit constant can be taken independent of $k$. Note that the exponent is negative for all $\gamma \in (\sqrt{2},2)$ and $p, \epsilon > 0$ sufficiently small, where $\epsilon$ is such that $p p' = p+\epsilon$.
\end{proof}

\begin{lem}
\label{lem:locally_finite_disk}
Fix $\kappa' \in (4,8)$, let $\gamma = 4/\sqrt{\kappa'}$, let $\eta^0 \sim \SLE_{\kappa'}$ and let $h^w$ be an independent weight $\tfrac{3\gamma^2}{2}-2$ quantum wedge parameterized by $\UH$ and with the circle average embedding.  Let $h^w = h^0 + \Fh$ be the decomposition of $h^w$ into a GFF $h^0$ on $\UH$ with zero boundary conditions and a function $\Fh$ which is harmonic on $\UH$.  Then we a.s.\ have that $\E[\sigma^{\eta^0,h^w}(\D \cap \UH) \giv \Fh] < \infty$. 
\end{lem}
\begin{proof}
Let $\gamma = 4/\sqrt{\kappa'}$ and fix $\epsilon, p > 0$ small enough so that the bounds of Lemma~\ref{lem:rectangle_bound} hold with $\wt{p} = p+\epsilon$.

We parameterize $\eta$ by generalized quantum length (as we shall bound the mass of $\eta \cap \D$ under the measure $\mu^0$, the parameterization of $\eta$ does not matter) and let $\tau = \sup\{ t \geq 0: \eta(t) \in \D \}$ and $R = \sup_{0 \leq t \leq \tau} |\eta(t)|$. We note that since $\eta$ is parameterized by generalized quantum length, $\sigma^{\eta^0,h^w}( \D \cap \UH) \leq \tau$ and hence it suffices to show that $\E[\tau \giv \Fh] <\infty$ a.s.

The strategy to bound $\E[\tau \giv \Fh]$ is as follows. We first define an event $E$ for $\Fh$ which can be taken to have probability as close to $1$ as we want.  We note that if $\tau$ is large, then the mass of $\mu_{h^w}(B(0,R))$ is typically large as well. Morally, this is because $\eta$ is then likely to cut out many quantum disks of boundary length at least $1$, of which many will have mass at least $1$ as well. Hence, the probability that this happens will be small due to bounds on $\mu_{h^w}(B(0,R))$.

For now, change the coordinates of the quantum surface $(\UH,h^w)$ to $(\SSS,\wt{h}^w)$ (recall~\eqref{eq:QS}). Let $\wt{h}^w = \wt{h}^0 + \wt{\Fh}$ be the decomposition of $\wt{h}^w$ into a sum of a GFF $\wt{h}^0$ on $\SSS$ with zero boundary conditions and a function $\wt{\Fh}$ which is harmonic on $\SSS$.  Fix a constant $C > 0$.  For each $k \geq 0$ we let $\wt{E}_k = \{\E[\mu_{\wt{h}^w}(A_k) \giv \wt{\Fh}] \leq C e^{(1+\epsilon)k(p(\gamma^2-2) + \wt{p}^2 \gamma^2)/p} \}$.  By Chebyshev's inequality and Lemma~\ref{lem:rectangle_bound} we have that
\begin{align*}
	\p[\wt{E}_k^c] \lesssim C^{-p} e^{-\epsilon k(p(\gamma^2-2) + \wt{p}^2 \gamma^2)}.
\end{align*}
As the exponent above is negative for $p,\epsilon > 0$ sufficiently small, by a union bound we have that $\cap_{k \geq 0} \wt{E}_k$ can be made to have probability as close to $1$ as we like by choosing $C > 0$ large.  For each $k \leq -1$ we similarly let $\wt{E}_k = \{ \E[\mu_{\wt{h}^w}(A_k) \giv \wt{\Fh}] \leq C e^{-(1-\epsilon)|k|(p(\gamma^2-2) - \wt{p}^2 \gamma^2)/p} \}$.  By Chebyshev's inequality and Lemma~\ref{lem:rectangle_bound} we have in this case that
\begin{align*}
	\p[\wt{E}_k^c] \lesssim C^{-p} e^{-\epsilon |k|(p(\gamma^2-2) - \wt{p}^2 \gamma^2)}.
\end{align*}
As the exponent above is negative for $p,\epsilon > 0$ sufficiently small, by a union bound we have that $\cap_{k \leq -1} \wt{E}_k$ can be made to have probability as close to $1$ as we like by choosing $C > 0$ large.  Altogether, with $\wt{E} = \cap_k \wt{E}_k$ by choosing $C > 0$ large we can make $\wt{E}$ to have probability as close to $1$ as we like.  Changing the parameterization back to $(\UH,h^w)$, and letting $E$ be the analog of $\wt{E}$ in $\UH$, that is, $E = \cap_{k} E_k$ where $E_k = \{ \E[\mu_{h^w}(\wh{A}_k) \giv \Fh] \leq C e^{(1+\epsilon)k(p(\gamma^2-2) + \wt{p}^2 \gamma^2)/p} \}$ for $k \geq 0$ and $E_k = \{ \E[\mu_{h^w}(\wh{A}_k) \giv \Fh] \leq C e^{-(1-\epsilon)(|k|p(\gamma^2-2) - \wt{p}^2 \gamma^2)/p} \}$ for $k \leq -1$ where $\wh{A}_k = B(0,e^{k+1}) \setminus B(0,e^k)$ then by the above and that $E \in \sigma(\Fh)$ we have
\begin{align*}
	\E[ \mu_{h^w}(B(0,e^k)) \one_{E} ] \lesssim e^{(1+\epsilon)k(p(\gamma^2-2) + \wt{p}^2 \gamma^2)/p} \quad\text{for}\quad k \geq 0.
\end{align*}
We now focus on $\eta \sim \SLE_{\kappa'}$. Note that by~\cite[Theorem~1.1, (1.3)]{wz2017boundaryarm} with $n=2$, we have for $k \geq 1$ that
\begin{align*}
	\p[\eta \ \text{returns to} \ \D \ \text{after leaving} \ B(0,e^k)] \asymp e^{-\big(\frac{3 \gamma^2}{4}-1 + o(1) \big) k}.
\end{align*}
Thus, recalling that $\tau$ is the last time $\eta$ visits $\D$ and $R = \sup_{0 \leq t \leq \tau} |\eta(t)|$, we have that
\begin{align}\label{eq:mass_bound_ball}
	\E[ \mu_{h^w}(B(0,R)) \one_E ] &\leq \sum_{k = 1}^\infty \E[ \mu_{h^w}(B(0,e^{k+1})) \one_E ] \p[ \eta \ \text{visits} \ \D \ \text{after leaving} \ B(0,e^k) ] \\
	&= \sum_{k = 1}^\infty O\Big( e^{(1+\epsilon)k(p(\gamma^2-2) + \wt{p}^2 \gamma^2)/p} \cdot e^{-\big(\frac{3 \gamma^2}{4}-1 + o(1) \big) k} \Big). \nonumber
\end{align}
By choosing $\epsilon, p > 0$ sufficiently small we can ensure that
\begin{align*}
	\frac{1+\epsilon}{p} (p(\gamma^2-2) + \wt{p}^2 \gamma^2) -\left(\frac{3 \gamma^2}{4}-1 \right) < 0
\end{align*}
so that the sum in~\eqref{eq:mass_bound_ball} converges and hence $\E[ \mu_{h^w}(B(0,R)) \one_E ] < \infty$.

Recall that since $\eta$ is parameterized by generalized quantum length, the left/right boundary lengths are independent $\kappa'/4$-stable L\'{e}vy processes.  It thus follows that there is a constant $c > 0$ such that the number of downward jumps of size at least $1$ made before time $T$ has law $\Po(cT)$. By Poisson concentration, there is a constant $\wt{c} > 0$ such that the probability that at most $cT/2$ negative jumps of size at least $1$ have occurred by time $T$ is $O(e^{-\wt{c}T})$.

Each of the surfaces cut out by $\eta$ corresponding to a downward jump of either the left or right boundary length process is a quantum disk with boundary length equal to the size of the downward jump and conditionally independent of the others given its boundary length. Moreover, if a quantum disk has boundary length $\ell \geq 1$, then there is a constant $a \in (0,1)$ uniform in $\ell \geq 1$ such that the event that the quantum mass is at least $1$ has probability lower bounded by $a$. Consequently, by Hoeffding's inequality there exists a constant $c^* > 0$ such that the probability that $\eta$ cuts out fewer than $acT/4$ quantum disks with mass at least $1$ is $O(e^{-c^* T})$. It follows that there exists a constant $q > 0$ such that for each $r > 0$ we have
\begin{align*}
	\p[ \tau \geq q r M, M \leq \mu_{h^w}(B(0,R)) \leq M+1, E] = O(e^{-c^*rM}).
\end{align*}
By summing over $r$ we see that
\begin{align*}
	\E[\tau \one_{\tau \geq 2qM} \one_{M \leq \mu_{h^w}(B(0,R)) \leq M+1} \one_E] = O(e^{-2c^*M})
\end{align*}
and hence
\begin{align*}
	&\E[ \tau \one_{M \leq \mu_{h^w}(B(0,R)) \leq M+1} \one_E ] \\
	&= \E[\tau \one_{\tau < 2 qM} \one_{M \leq \mu_{h^w}(B(0,R)) \leq M+1} \one_E] + \E[\tau \one_{\tau \geq 2 qM} \one_{M \leq \mu_{h^w}(B(0,R)) \leq M+1} \one_E] \\
	&\leq 2qM \p[ M \leq \mu_{h^w}(B(0,R)) \leq M+1, E] + O(e^{-2c^*M}).
\end{align*}
Since $\E[ \mu_{h^w}(B(0,R)) \one_E] < \infty$, we get by summing over $M$ that
\begin{align*}
	\E[ \tau \one_E ] < \infty.
\end{align*}
Since $E \in \sigma(\Fh)$, this implies that $\E[ \tau \giv \Fh] \one_E < \infty$ a.s.  Since we can make the probability of $E$ as close to $1$ as we like by adjusting the value of $C > 0$, we obtain that $\E[\tau \giv \Fh] < \infty$ a.s.  This completes the proof since $\sigma^{\eta^0,h^w}( \D \cap \UH) \leq \tau$ as we explained above.
\end{proof}

\begin{proof}[Proof of Lemma~\ref{lem:locally_finite_SLE}]
We begin by deducing from Lemma~\ref{lem:locally_finite_disk} that $\E[ \sigma^{\eta^0,h^0}](K) < \infty$ for each compact $K \subseteq \D \cap \UH$. 

First we note that since adding a constant $C$ to a field $h^*$ scales quantum lengths by $e^{\frac{\gamma}{2} C}$, for any $U \subseteq \UH$ we have that
\begin{align}\label{eq:scaling_quantum_length}
	\sigma^{\eta^0,h^*+C}(U) &= \lim_{\epsilon \to 0} \epsilon^{\wh{\alpha}} N_\epsilon^{\eta^0,h^*+C}(U) = \lim_{\epsilon \to 0} \epsilon^{\wh{\alpha}} N_{\epsilon e^{-\frac{\gamma}{2} C}}^{\eta^0,h^*}(U) = e^{\wh{\alpha}\frac{\gamma}{2}C} \lim_{\epsilon \to 0} (\epsilon e^{-\frac{\gamma}{2}C})^{\wh{\alpha}} N_{\epsilon e^{-\frac{\gamma}{2} C}}^{\eta^0,h^*}(U) \\
	&= e^{\wh{\alpha}\frac{\gamma}{2}C} \sigma^{\eta^0,h^*}(U) = e^{\frac{2}{\gamma} C} \sigma^{\eta^0,h^*}(U). \nonumber
\end{align}

Fix some compact set $K \subseteq \D \cap \UH$ and let $\delta \coloneqq \dist(K, \partial \UH) > 0$. We decompose $h^w$ into $h^w = h^0 + \Fh$, where $h^0$ is a zero-boundary GFF on $\UH$, $\Fh$ is a distribution which is harmonic on $\UH$, and $h^0$, $\Fh$ are independent. Then $\Fh$ is a.s.\ bounded on $K$ since $\delta > 0$ and hence $\p[ \sup_{z \in K} | \Fh(z) | \leq C] \to 1$ as $C \to \infty$. Thus letting $E_K$ be the event that $\sup_{z \in K} |\Fh(z)| \leq C$ we can choose $C < \infty$ so that $\p[E_K] \geq 1/2$.  It follows from~\eqref{eq:scaling_quantum_length} that
\begin{align*}
	\E[ \sigma^{\eta^0,h^0}(K) \giv \Fh] \one_{E_K} \leq e^{\frac{2}{\gamma} C} \E[ \sigma^{\eta^0,h^w}(K) \giv \Fh] \one_{E_K} \leq e^{\frac{2}{\gamma} C} \E[\sigma^{\eta^0,h^w}(K) \giv \Fh].
\end{align*}
Lemma~\ref{lem:locally_finite_disk} implies that the right hand side is a.s.\ finite and the independence of $h^0$ and $\Fh$ implies that $\E[ \sigma^{\eta^0,h^0}(K) \giv \Fh] = \E[ \sigma^{\eta^0,h^0}(K)]$.  Altogether, we have proved that $\E[\sigma^{\eta^0,h^0}(K)] < \infty$.

Finally, since $r_{\UH}(z) \geq \dist(z,\partial \UH)$, it follows that
\begin{align*}
	\E[\mu^0(K)] \leq \delta^{-2/\gamma^2} \E[\sigma^{\eta^0,h^0}(K)] < \infty.
\end{align*}
We now turn to the case of a general compact subset of $\UH$. Fix some compact set $\wt{K} \subseteq \UH$. Let $R' = \inf\{ r > 0: \wt{K} \subseteq B(0,r)\}$ and let $R = 2R'$. Let $\varphi_R:\UH \to \UH$ be given by $\varphi_R(z) = z/R$. Then $K \coloneqq \varphi_R(\wt{K}) \subseteq \D \cap \UH$ is compact. Moreover, by~\eqref{eq:QS}, we a.s.\ have that
\begin{align}\label{eq:coordinate_change_quantum_length}
	\sigma^{\varphi_R(\eta^0),h^0 \circ \varphi_R^{-1} + Q\log| (\varphi_R^{-1})'|}(K) = \sigma^{\eta^0,h^0}(\wt{K}).
\end{align}
Furthermore, $\varphi_R(\eta^0) \sim \SLE_{\kappa'}$ and $h^0 \circ \varphi_R^{-1}$ is a zero-boundary GFF on $\UH$. Thus $\sigma^{\eta^0,h^0}$ and $\sigma^{\varphi_R(\eta^0),h^0 \circ \varphi_R^{-1}}$ have the same law, so by~\eqref{eq:scaling_quantum_length} it follows that
\begin{align*}
	\E\!\left[ \sigma^{\eta^0,h^0}\right]\!(\wt{K}) &= \E\!\left[\sigma^{\varphi_R(\eta^0),h^0 \circ \varphi_R^{-1} + Q\log| (\varphi_R^{-1})'|}\right]\!(K) \\
	&= R^{\frac{2}{\gamma} Q} \E\!\left[\sigma^{\varphi_R(\eta^0),h^0 \circ \varphi_R^{-1}}\right]\!(K) \\
	&= R^{1+\frac{4}{\gamma^2}} \E\!\left[ \sigma^{\eta^0,h^0}\right]\!(K) < \infty,
\end{align*}
since $(\varphi_R^{-1})' = R$, $\tfrac{2}{\gamma} Q = 1 + \tfrac{4}{\gamma^2}$ and $K \subseteq \D \cap \UH$. Noting that, as above,
\begin{align*}
	\E\!\left[ \mu^0 \right]\!(\wt{K}) \leq \dist(\wt{K},\R)^{-2/\gamma^2} \E\!\left[ \sigma^{\eta^0,h^0} \right]\!(\wt{K}) < \infty,
\end{align*}
the result follows.
\end{proof}

\subsection{Conformal covariance and uniqueness}
\begin{lem}
\label{lem:measure_SLE}
We have that $\wt{\mu}^s = \mu^s$, i.e.,
\begin{align*}
    \mu^0|_{\eta^0([s,\infty))} \circ \psi_s(dz) = |\psi_s'(z)|^{1+\frac{\kappa'}{8}} \wt{\mu}^s(dz).
\end{align*}
\end{lem}
\begin{proof}
Since quantum length is invariant under the coordinate change $h^0 \mapsto h^0 \circ \psi_s + Q\log|\psi_s'|$, we have
\begin{align*}
    \mu^0|_{\eta^0([s,\infty))} \circ \psi_s(dz) &= F(\psi_s(z)) \E\!\left[\sigma_{[s,\infty)}^{\eta^0,h^0} \circ \psi_s(dz) \, \middle| \, \eta^0 \right] \\
    &= F(\psi_s(z)) \E\!\left[ \lim_{\epsilon \rightarrow 0} \epsilon^{\wh{\alpha}} N_\epsilon^{\eta^0,h^0,[s,\infty)} \circ \psi_s(dz) \, \middle| \, \eta^0 \right] \\
    &= F(\psi_s(z)) \E\!\left[ \lim_{\epsilon \rightarrow 0} \epsilon^{\wh{\alpha}} N_\epsilon^{\eta^s,h^0\circ \psi_s + Q\log|\psi_s'|}(dz) \, \middle| \, \eta^0 \right]\!.
\end{align*}
Furthermore, since $\wt{h}^s$ is a zero-boundary GFF, we have that $\wt{h}^s \circ f_s$ is a zero-boundary GFF in $\UH \setminus \eta^0([0,s])$. Thus we can assume that $h^0$, $\wt{h}^s$ are coupled together in such a way so that we can define a Gaussian field $\FH$ which conditionally on $\eta^0$ is independent of $\wt{h}^s$ and is such that $h^0 = \wt{h}^s \circ f_s + \FH$. Then $\FH$ has covariance kernel given by
\begin{align*}
    G_{\UH}(z,w)- G_{\UH\setminus \eta^0([0,s])}(z,w) = G_{\UH}(z,w) - G_{\UH}(f_s(z),f_s(w))
\end{align*}
and at a point $z$ the variance of $\FH(z)$ is given by 
\begin{align}\label{eq:C0var}
    \textup{Var} \, \FH(z) = \lim_{\epsilon \rightarrow 0} \textup{Var} \, \FH_\epsilon (z) = \log r_{\UH}(z) - \log r_{\UH}(f_s(z)) + \log |f_s'(z)|,
\end{align}
where $\FH_\epsilon$ denotes the circle average of $\FH$ (the $\log |f_s'(z)|$ term comes from the change of variables dilating the ball of radius $\epsilon$ by a factor of $|f_s'(z)|$ as $\epsilon \rightarrow 0$). Since $\wt{h}^s = h^0 \circ \psi_s - \FH \circ \psi_s$ and arguing as in~\eqref{eq:scaling_quantum_length} it follows that
\begin{align*}
    &\lim_{\epsilon \rightarrow 0} \epsilon^{\wh{\alpha}} N_\epsilon^{\eta^s,h^0 \circ \psi_s + Q\log|\psi_s'|}(dz) \\
    &= \lim_{\epsilon \rightarrow 0} \epsilon^{\wh{\alpha}} N_{\epsilon |\psi_s'|^{-Q\frac{\gamma}{2}}}^{\eta^s,h^0 \circ \psi_s}(dz) \\
    &= \lim_{\epsilon \rightarrow 0} \epsilon^{\wh{\alpha}} N_{\epsilon |\psi_s'|^{-Q\frac{\gamma}{2}} e^{-\frac{\gamma}{2} \FH \circ \psi_s}}^{\eta^s,\wt{h}^s}(dz)\\
    &= |\psi_s'(z)|^{\wh{\alpha} Q\frac{\gamma}{2}} e^{\wh{\alpha} \frac{\gamma}{2} \FH \circ \psi_s(z)} \lim_{\epsilon \rightarrow 0} (\epsilon |\psi_s'(z)|^{-Q\frac{\gamma}{2}} e^{-\frac{\gamma}{2} \FH \circ \psi_s(z)})^{\wh{\alpha}} N_{\epsilon |\psi_s'|^{-Q\frac{\gamma}{2}} e^{-\frac{\gamma}{2} \FH \circ \psi_s}}^{\eta^s,\wt{h}^s}(dz) \\
    &= |\psi_s'(z)|^{\wh{\alpha} Q\frac{\gamma}{2}} e^{\wh{\alpha}\frac{\gamma}{2} \FH \circ \psi_s(z)} \sigma^{\eta^s,\wt{h}^s}(dz).
\end{align*}
Since $\FH$ is Gaussian with variance given by~\eqref{eq:C0var} and $\wh{\alpha} \gamma/2 = 2/\gamma$, we have that
\begin{align*}
    \E\!\left[ e^{\wh{\alpha} \frac{\gamma}{2} \FH \circ \psi_s(z)} \, \middle| \, \eta^0 \right] &= e^{\frac{2}{\gamma^2} (\log r_{\UH}(\psi_s(z)) - \log r_{\UH}(z) + \log |f_s'(\psi_s(z))|)} \\
    &= r_{\UH}(\psi_s(z))^{2/\gamma^2} r_{\UH}(z)^{-2/\gamma^2} |\psi_s'(z)|^{-2/\gamma^2} \\
    &= \frac{F(z)}{F(\psi_s(z))} |\psi_s'(z)|^{-2/\gamma^2}.
\end{align*}
Thus, noting that $\wh{\alpha} Q \frac{\gamma}{2} = 1 + \frac{4}{\gamma^2}$, it follows that 
\begin{align*}
    \mu^0|_{\eta^0([s,\infty))} \circ \psi_s(dz) &= F(\psi_s(z)) \E\!\left[ |\psi_s'(z)|^{\wh{\alpha} Q\frac{\gamma}{2}} e^{\wh{\alpha}\frac{\gamma}{2} \FH \circ \psi_s(z)} \sigma^{\eta^s,\wt{h}^s}(dz) \, \middle| \, \eta^0 \right] \\
    &= |\psi_s'(z)|^{1+ \frac{4}{\gamma^2}} F(\psi_s(z)) \E\!\left[ e^{\wh{\alpha}\frac{\gamma}{2} \FH \circ \psi_s(z)} \, \middle| \, \eta^0 \right] \E\!\left[ \sigma^{\eta^s,\wt{h}^s}(dz) \, \middle| \, \eta^0 \right] \\
    &= |\psi_s'(z)|^{1+\frac{2}{\gamma^2}} F(z) \E\!\left[ \sigma^{\eta^s,\wt{h}^s}(dz) \, \middle| \, \eta^s \right] \\
    &= |\psi_s'(z)|^{1+\frac{\kappa'}{8}} \wt{\mu}^s(dz),
\end{align*}
as was to be shown.
\end{proof}

\begin{rmk}
Define for each $b>0$ the function $\phi_b(z) = bz$. Following the same strategy as in Lemma~\ref{lem:measure_SLE} with $\phi_b$ in place of $\psi_s$, it is easy to see that if $\wh{\eta}(t) = b\eta(t/b^2)$, $\wh{h}$ is a zero-boundary GFF independent of $\eta$ and we let $\wh{\mu}(dz) = F(z)\E\!\left[ \sigma^{\wh{\eta},\wh{h}}(dz) \, \middle| \, \wh{\eta} \right]$, then $\mu$ and $\wh{\mu}$ have the same law and satisfy
\begin{align}\label{eq:conformal_covariance_scalar}
    \mu|_{\eta([s,t])} \circ \phi_b(dz) = b^{1+\frac{\kappa'}{8}} \wh{\mu}|_{\wh{\eta}([s/b^2,t/b^2])}(dz).
\end{align}
This, in turn, gives us that
\begin{align}\label{eq:conformal_covariance_intensity_scalar}
    \E[\mu] \circ \phi_b(dz) = b^{1+\frac{\kappa'}{8}} \E[\mu](dz)
\end{align}
for all $b > 0$.
\end{rmk}

We now turn to proving that $\mu^0$ is locally finite in the sense that we a.s.\ have that $\mu^0(\eta^0([s,t])) < \infty$ for all $0 \leq s < t < \infty$. We will accomplish this by first showing that $\E[\mu^0]$ has a density with respect to Lebesgue measure and then deducing the form of it.  We note that Lemma~\ref{lem:measure_SLE} implies that the assumptions of Theorem~\ref{thm:SLE_unique} are satisfied except for local finiteness.  If we knew that $\mu^0$ was locally finite, then Theorem~\ref{thm:SLE_unique} would apply and we would have that the density of $\E[\mu^0]$ with respect to Lebesgue measure is a positive multiple of the function $G_\kappa$ defined in Section~\ref{sec:preliminaries_natural_parameterization} as $G_\kappa$ is the density of the intensity of the natural parameterization.  The way that we will reason in what follows is that after showing that $\E[\mu^0]$ has a density with respect to Lebesgue measure we will prove that it is locally finite by showing that its density must be a multiple of $G_\kappa$.  At this point we will be able to apply Theorem~\ref{thm:SLE_unique} to see that $\mu^0$ must be (a multiple of) the natural parameterization.

\begin{lem}
\label{lem:intensity_SLE_continuity}
We have that $\E[\mu^0]$ is absolutely continuous with respect to Lebesgue measure.
\end{lem}
\begin{proof}
Suppose that $A \subseteq \UH$ has zero Lebesgue measure.  We want to show that $\E[\mu^0](A) = 0$.  It suffices to show that $\E[ \mu^0](A \cap (\UH + ir)) =0$ for each $r > 0$.  We may therefore assume that $A \subseteq \UH+ir$.  We will prove the result in the case that $r=2$.  The result for other values of $r > 0$ follows from the same argument.  By \cite[Lemma~1]{ll2009geo} we have that $2 = \hcap(\eta([0,1])) \geq \sup\{\im(z) : z \in \eta([0,1])\}^2/2$.  That is, $\sup\{\im(z) : z \in \eta([0,1])\} \leq 2$.  Therefore $f_t(z)$ is defined for all $z \in A$ and $t \in [0,1]$.  Let~$U$ be uniform in $[0,1]$ independently of everything else.  By~\eqref{eq:confcovSLE} we have that
\[ \E[ \mu^0](A) = \E\left[ \int |\psi_U'(z)|^{d} \one_{z \in f_U(A)} d\mu^U(z) \right] = \int \E[ |\psi_U'(z)|^{d} \one_{z \in f_U(A)}] d\E[\mu^0(z)]\]
where in second equality we used that $\mu^U \stackrel{d}{=} \mu^0$ and that $\mu^U$ is independent of $f_U$.  To show that the right hand side is equal to zero, it suffices to show that $\p[ \psi_U(z) \in A] = 0$ for each $z \in \UH$.  Let $(\wt{f}_t)$ solve the centered reverse Loewner equation with driving function $\wt{W}_t = \sqrt{\kappa} \wt{B}_t$ where $\wt{B}$ is a standard Brownian motion.  That is,
\[ d\wt{f}_t(z) = -\frac{2}{\wt{f}_t(z)} dt - d\wt{W}_t,\quad \wt{f}_0(z) = z.\]
Then we have that $\psi_t(z) \stackrel{d}{=} \wt{f}_t(z)$ for all $t$.  Consequently, it suffices to show that $\p[ \wt{f}_U(z) \in A] = 0$ for each $z \in \UH$.  Let $\wt{Z}_t = \wt{f}_t(z)$, $\wt{\Theta}_t = \arg \wt{Z}_t$, and $\wt{I}_t = \log \im (\wt{Z}_t)$.  We will complete the proof by determining the SDE solved by $\wt{Z}_t$ after performing a certain time change (which will follow very closely from the calculations given in the proof of \cite[Proposition~3.8]{dms2014mating}).   We have that
\begin{equation}
d \wt{Z}_t = -\frac{2}{\wt{Z}_t} dt - d\wt{W}_t.
\end{equation}
Then we have that
\[ d \im (\wt{Z}_t) = - \im\left(\frac{2}{\wt{Z}_t}\right)dt.\]
This implies that
\[ d \wt{I}_t = d \log \im( \wt{Z}_t ) = -\frac{1}{\im(\wt{Z}_t)} \im\left(\frac{2}{\wt{Z}_t} \right) = \frac{2}{|\wt{Z}_t|^2} dt.\]
We also have that
\begin{align*}
  d \log \wt{Z}_t &= -\frac{\sqrt{\kappa}}{\wt{Z}_t} d\wt{B}_t - \frac{2+\kappa/2}{\wt{Z}_t^2} dt \quad\text{and}\\
  d \wt{\Theta}_t &= -\im\left( \frac{\sqrt{\kappa}}{\wt{Z}_t}\right) d\wt{B}_t - \im\left( \frac{2+\kappa/2}{\wt{Z}_t^2} \right) dt.
\end{align*}
We now reparameterize time by letting $ds = |\wt{Z}_t|^{-2} dt$.  Then we have that $d \wt{I}_{t(s)} = 2 ds$ and there exists a Brownian motion $\wh{B}$ so that $d\wh{B}_s = d\wt{B}_t/|\wt{Z}_t|$.  Altogether, we have that
\[ d \wt{\Theta}_{t(s)} = \sqrt{\kappa} \sin(\wt{\Theta}_{t(s)}) d \wh{B}_s + \left(2+ \frac{\kappa}{2}\right) \sin(2\wt{\Theta}_{t(s)}) ds.\]
In particular, it is not difficult to see for $s > 0$ that the law of $\wt{\Theta}_{t(s)}$ has a density with respect to Lebesgue measure.  Thus if $S \sim \exp(1)$ independently of everything else then the law of $\wt{Z}_{t(S)}$ has a density with respect to Lebesgue measure so that $\p[ \wt{Z}_{t(S)} \in A] = 0$.  Note that the conditional law of
\[ t^{-1}(U) = \int_0^U \frac{1}{|\wt{Z}_u|^2} du\]
given $\wt{W}$ is absolutely continuous with respect to Lebesgue measure on $\R_+$ as $s \mapsto t^{-1}(s) = \int_0^s |\wt{Z}_u|^{-2} du$ is a.s.\ strictly increasing and is a.s.\ $C^1$.  Therefore the conditional law of $t^{-1}(U)$ given~$\wt{W}$ is absolutely continuous with respect to the law of~$S$ (which is the same as the law of~$S$ given~$\wt{W}$ by independence).  Therefore the law of the pair $(\wt{W},t^{-1}(U))$ is absolutely continuous with respect to the law of the pair $(\wt{W},S)$.  Since $\wt{Z}_U$ is obtained from $(\wt{W},t^{-1}(U))$ by applying the same measurable function which constructs $\wt{Z}_{t(S)}$ from $(\wt{W},S)$, this implies that the law of $\wt{Z}_U$ is absolutely continuous with respect to the law of $\wt{Z}_{t(S)}$.  Therefore we also have that $\p[\wt{Z}_U \in A] = 0$.
\end{proof}

In the next lemma, we note that by Lemma~\ref{lem:intensity_SLE_continuity}, \eqref{eq:conformal_covariance_intensity_scalar} determines the structure of the radial part of~$\mu^0$.

\begin{lem}
For each $b>0$, let $\phi_b(z) = bz$. Let $m$ be a measure on $\UH$ which is absolutely continuous with respect to Lebesgue measure and satisfies
\begin{align*}
    m \circ \phi_b(dz) = b^d m(dz)
\end{align*}
for some $d>1$ and all $b > 0$. Then there exists some function $H(z) = H(\arg z)$ such that
\begin{align*}
    m(dz) = H(\arg z) \im(z)^{d-2} dz.
\end{align*}
\end{lem}
\begin{proof}
Let $f(z)$ denote the density of $m$ with respect to Lebesgue measure. Then
\begin{align*}
    \frac{d(m \circ \phi_b)}{dz}(z) = f(\phi_b(z)) |\phi_b'(z)|^2 = b^2 f(bz)
\end{align*}
and consequently
\begin{align}\label{eq:RN_1}
    \frac{d(m \circ \phi_b)}{dm}(z) = \frac{\frac{d(m \circ \phi_b)}{dz}(z)}{\frac{dm}{dz}(z)} = \frac{b^2 f(bz)}{f(z)}.
\end{align}
Moreover, by the statement of the lemma, we have that
\begin{align}\label{eq:RN_2}
    \frac{d(m \circ \phi_b)}{dm}(z) = b^d.
\end{align}
By combining \eqref{eq:RN_1} and \eqref{eq:RN_2}, we have that
\begin{align*}
    f(bz) = b^{d-2} f(z).
\end{align*}
This implies that $f(z) = \wh{H}(\arg z)|z|^{d-2}$ for some function $\wh{H}$. Since $\im(z) = |z| \sin(\arg z)$ this proves the lemma.
\end{proof}

We now turn to determining the form of the function $H$ in the case $m = \E[\mu^0]$.

\begin{prop}
\label{prop:density_intensity_SLE}
There is some constant $c>0$ such that
\begin{align*}
    \E[\mu^0(dz)] = c
\sin^{\frac{8}{\kappa}-1}(\arg z) \im(z)^{d-2} dz.
\end{align*}
\end{prop}
\begin{proof}
Denote by $\Phi(z) = H(\arg z) \im(z)^{d-2}$ the density of $\E[\mu^0]$ with respect to Lebesgue measure. As above, let $f_t(z) = X_t^z + i Y_t^z$.  Let $\Upsilon_t^z = Y_t^z/|f_t'(z)|$ and note that $\Upsilon_t^z$ is half the conformal radius of $z$ in $\UH \setminus \eta([0,t])$, which is decreasing in $t$.  Let $\Theta_t^z = \arg f_t(z)$. By~\eqref{eq:confcovSLE} we have for all $z \in \UH$ that
\begin{align}
\label{eqn:mu_meas_cond_exp}
    \E[ \mu^0(dz) \giv \eta|_{[0,t]} ] = |f_t'(z)|^{2-d} \Phi(f_t(z)) dz = H(\Theta_t^z) (\Upsilon_t^z)^{d-2} dz.
\end{align}
We note that the filtration generated by $\eta$ is the same as the filtration generated by the driving Brownian motion.  In particular, the right hand side above is a martingale adapted to a filtration generated by a Brownian motion and therefore has a continuous modification.  We shall assume that we are working with this modification.  Since the process $(\Upsilon_t^z)^{d-2}$ is continuous in $t$ and positive (up until $z$ is swallowed) it follows that $t \mapsto H(\Theta_t^z)$ is also continuous (up until $z$ is swallowed).

We will argue that $H|_{[\arg z,\pi)}$ has a continuous version; the same argument implies that $H|_{(0,\arg z]}$ has a continuous version (and the two versions agree at $\arg z$ so $H$ has a continuous version).  Let $\tau$ be the time that $z$ is swallowed and assume that we are working on the event that $\Theta_\tau^z = \pi$.  We first note that there are a.s.\ only countably many times $t_0$ so that $\Theta^z$ has a local minimum or maximum at the time $t_0$.  Then as $\Theta^z$ a.s.\ hits every $\theta \in [\arg z,\pi)$ uncountably many times (\cite[Theorem~6.40]{mp2010brownian}, using that $\Theta^z$ is up to time change locally absolutely continuous with respect to a standard Brownian motion) it is a.s.\ the case that for every $\theta \in [\arg z,\pi)$ there is a time $t_\theta$ so that $\Theta_{t_\theta}^z = \theta$ and $\Theta^z$ does not have a local minimum or maximum at the time $t_\theta$.  Let $t_\theta'$ be any other time so that $\Theta_{t_\theta'}^z = \theta$.  We claim that $H_{t_\theta} = H_{t_\theta'}$.  To see this, we first note that it is a.s.\ the case that $\Theta^z$ is not constant on any open interval of time.  Thus by the continuity of $\Theta^z$ and since $H_t = H(\Theta_t^z)$ for a.e.\ $t \in [0,\tau]$ we can find sequences $(a_n)$, $(a_n')$ so that $a_n \to t_\theta$ and $a_n' \to t_\theta'$, $\Theta_{a_n}^z = \Theta_{a_n'}^z$ for all $n \in \N$, and $H_{a_n} = H(\Theta_{a_n}^z)$, $H_{a_n'} = H(\Theta_{a_n'}^z)$ for all $n \in \N$.  That is, $H_{a_n} = H_{a_n'}$ for all $n \in \N$ hence $H_{t_\theta} = H_{t_\theta'}$.  For $\theta \in [\arg z,\pi)$ we then define $\wt{H}(\theta) = H_{t_\theta}$.  Then we have that $\wt{H}(\Theta_t^z) = H_t$ for all $t \geq 0$ and, arguing as above, it is not difficult to see that $\wt{H}$ is continuous.  We also note that $\wt{H}$ is not random.  Indeed, suppose that $(H_t^1, \Theta_t^{z,1})$, $(H_t^2, \Theta_t^{z,2})$ are independent copies of $(H_t, \Theta_t^z)$ conditioned on $\Theta_t^{z,i}$ hitting $\pi$ before $0$.  Then if we define $t_\theta^i$ for $(H_t^i, \Theta_t^{z,i})$, $i=1,2$, in the same way that we defined $t_\theta$ for $(H_t, \Theta_t^z)$ we a.s.\ have that $H_{t_\theta^1}^1 = H_{t_\theta^2}^2$ for all $\theta \in [\arg z,\pi)$.  That is, the function $\wt{H}$ produced from $(H_t^1, \Theta_t^{z,1})$ a.s.\ agrees with the function produced from $(H_t^2, \Theta_t^{z,2})$.  It is left to show that $\wt{H}$ is a version of $H |_{[\arg z,\pi)}$, i.e., $\wt{H}(\theta)  = H(\theta)$ for a.e. \ $\theta \in [\arg z,\pi)$.  We note that $H_1 = \wt{H}(\Theta_1^z)$ is equal to $H(\Theta_1^z)$ with probability $1$.  As the law of $\Theta_1^z$ has a positive density on $[\arg z,\pi)$ with respect to Lebesgue measure, we conclude that $\wt{H}$ is a version of $H|_{[\arg z,\pi)}$.

Since the measure defined by the density $H(\arg z) \im(z)^{d-2}$ with respect to Lebesgue measure is not affected by changing $H$ on a set of Lebesgue measure zero we may thus assume that $H$ is continuous.  Fix some $\delta > 0$ and let $\tau_z = \tau_z^\delta = \inf \{t \geq 0: \Upsilon_t^z \leq \delta \}$.  Fix $r, t > 0$ and let $\sigma_z = \tau_z \wedge \inf\{s \geq 0: H(\Theta_s^z) \geq r\}$.  Note that this is a stopping time by the continuity of $s \mapsto H(\Theta_s^z)$.  Let also $\sigma_z(t) = t \wedge \sigma_t$.

Taking expectations of~\eqref{eqn:mu_meas_cond_exp} at the time $\sigma_z(t)$ we have that
\begin{align}\label{eq:density_expectation1}
    \Phi(z) = \E[ H(\Theta_{\sigma_z(t)}^z) (\Upsilon_{\sigma_z(t)}^z)^{d-2} ].
\end{align}
Next, we consider the continuous local martingale \cite{sw2005coordinate}
\begin{align*}
    M_t^z = (Y_t^z)^{d-2} \sin(\Theta_t^z)^{\frac{8}{\kappa}-1} |f_t'(z)|^{2-d} = (\Upsilon_t^z)^{d-2} \sin(\Theta_t^z)^{\frac{8}{\kappa}-1}.
\end{align*}
Let $\p_z^*$ denote the measure obtained by weighting the underlying probability measure $\p$ by $M_t^z$. By \cite{sw2005coordinate}, under the measure $\p_z^*$, $\eta$ has the law of an $\SLE_\kappa(\kappa-8)$ process in $\UH$ with force point at $z$.  Note that if we denote by $\E_z^*$ the expectation with respect to $\p_z^*$, then \eqref{eq:density_expectation1} implies that
\begin{align}
\label{eqn:phi_formula}
    \Phi(z) &= \E\!\left[ H(\Theta_{\sigma_z(t)}^z) (\Upsilon_{\sigma_z(t)}^z)^{d-2}\right] = M_0^z \E_z^*\!\left[H(\Theta_{\sigma_z(t)}^z)/\sin(\Theta_{\sigma_z(t)}^z)^{\frac{8}{\kappa}-1}\right].
\end{align}
Next, we shall prove that $\Phi$ is a smooth function. 

Let $a = 2/\kappa$.  We parameterize $\eta$ so that
\begin{align*}
    \partial_t g_t(z) = \frac{a}{g_t(z) - B_t}, \quad g_0(z) = z.
\end{align*}
Then under the measure $\p_z^*$, $\Theta_t^z$ satisfies
\begin{align*}
    d \Theta_t^z = \frac{2a X_t^z Y_t^z}{((X_t^z)^2 + (Y_t^z)^2)^2} dt + \frac{Y_t^z}{(X_t^z)^2+(Y_t^z)^2} dW_t,
\end{align*}
where $W$ is a $\p_z^*$-Brownian motion. Let $\wt{t}(s)$ be the solution to
\begin{align*}
    s = \int_0^{\wt{t}(s)} \frac{(Y_u^z)^2}{((X_u^z)^2 + (Y_u^z)^2)^2} du.
\end{align*}
Then $\wt{\Upsilon}_s^z = \Upsilon_{\wt{t}(s)}^z$ satisfies $d \wt{\Upsilon}_s^z = -2a \wt{\Upsilon}_s^z ds$, that is, $\wt{\Upsilon}_s^z = \wt{\Upsilon}_0^z e^{-2as} = \im(z) e^{-2as}$. Moreover, letting $\wt{\Theta}_s^z = \Theta_{\wt{t}(s)}^z$ we have that $\wt{\Theta}_s^z$ under $\p_z^*$ satisfies
\begin{align*}
    d \wt{\Theta}_s^z = 2a \cot(\wt{\Theta}_s^z)ds + d\wt{W}_s, \quad \wt{\Theta}_0^z = \arg z,
\end{align*}
where $\wt{W}_s$ is a standard Brownian motion. Thus, under $\p_z^*$, $\wt{\Theta}_s^z$ has the law of a radial Bessel process of dimension $4a+1$, and thus, by~\cite[Proposition~B.1]{zhan2016ergodicity}, has transition density given by
\begin{align*}
    p_s(x,y) = \sin(y) \sum_{n=0}^\infty \frac{(1-\cos^2(y))^{2a-1/2} C_n^{(2a)}(\cos(x)) C_n^{(2a)}(\cos(y))}{\int_{-1}^1 (1-u^2)^{2a-1/2} C_n^{(2a)}(u)^2 du} \exp\!\left(-\frac{n}{2}(n+4a)t \right),
\end{align*}
where $C_n^{(\alpha)}$ denotes the Gegenbauer polynomial with degree $n$ and index $\alpha$, see~\cite{emot1953functions2}.  Thus, the function $q(z) = p_s(\arg z,y)$ is smooth.  Note that $\rho=\kappa-8 < \kappa/2-2$ for $\kappa \in (4,8)$ so that $\tau_z^\delta$ is $\p_z^*$-a.s.\ finite.  We assume that $\delta \in (0,\Upsilon_0^z)$.  Taking a limit as $t \to \infty$ and then $\delta \to 0$ in~\eqref{eqn:phi_formula} and applying Fatou's lemma we see that
\begin{align*}
\Phi(z)
&= \liminf_{\delta \to 0} \liminf_{t \to \infty} M_0^z\E_z^*\!\left[H(\Theta_{\sigma_z(t)}^z)/\sin(\Theta_{\sigma_z(t)}^z)^{\frac{8}{\kappa}-1}\right]\\
&\geq M_0^z \E_z^*\!\left[\liminf_{\delta \to 0} \liminf_{t \to \infty} H(\Theta_{\sigma_z(t)}^z)/\sin(\Theta_{\sigma_z(t)}^z)^{\frac{8}{\kappa}-1}\right].
\end{align*}
Since $\sin(\theta)^{1-8/\kappa} \geq 1$ for $\kappa \in (4,8)$ and $\theta \in (0,\pi)$ we have that the right hand side is bounded from below by $r \wedge \sup_{\theta \in (0,\pi)} H(\theta)$.  By taking a limit as $r \to \infty$ we thus see that $H$ is bounded by $\Phi(z) / M_0^z$.  Altogether, this implies that $H(\wt{\Theta}_s^z) / \sin(\wt{\Theta}_s^z)^{\frac{8}{\kappa}-1}$ is a continuous $\p_z^*$-martingale which is in $L^p$ for some $p > 1$ (as $\sin(\wt{\Theta}_s^z)^{1-\frac{8}{\kappa}}$ has a finite $p$th $\p_z^*$-moment for some $p > 1$ and we have just argued that $H$ is bounded).  Thus we can take a limit as $t \to \infty$ and $r \to \infty$ to see that
\[ \Phi(z) = M_0^z \E_z^*\!\left[H(\wt{\Theta}_s^z)/\sin(\wt{\Theta}_s^z)^{\frac{8}{\kappa}-1}\right]\]
where $s$ is such that $\wt{\Upsilon}_s^z = \delta$; note that $s > 0$.  As $p_s(\arg z,y)$ is smooth, this implies that $\Phi(z)$ is smooth as well.  Since $\Phi$ and $\im(z)$ are smooth in $\UH$, so is $H(z)$.

Moreover, by~\eqref{eqn:mu_meas_cond_exp} we have that
\begin{align*}
    \Phi(f_t(z)) |f_t'(z)|^{2-d} = H(\Theta_t^z) (\Upsilon_t^z)^{d-2}
\end{align*}
is a continuous local martingale. By the symmetry of $\SLE$, as well as the zero-boundary GFF, it follows that $H(\tfrac{\pi}{2} + \theta) = H(\tfrac{\pi}{2} - \theta)$ for all $\theta \in (0,\tfrac{\pi}{2})$. Since $H$ is smooth, it follows that $H'(\tfrac{\pi}{2}) = 0$ and this together with It\^{o}'s formula (since $H(\Theta_t^z) (\Upsilon_t^z)^{d-2}$ is a local martingale), implies that
\begin{align}\label{eq:ODE}
    H''(\theta) + \left( 2 - \frac{8}{\kappa} \right) \cot(\theta)H'(\theta) + \left( \frac{8}{\kappa} - 1 \right) H(\theta) = 0, \quad H\!\left(\frac{\pi}{2} \right) = c, \quad H'\!\left(\frac{\pi}{2} \right) = 0,
\end{align}
for some $c>0$. By a standard uniqueness result for linear second-order differential equations, there exist two linearly independent solutions to~\eqref{eq:ODE}, namely $H \equiv 0$ and $H(\theta) = c \sin^{\frac{8}{\kappa}-1}(\theta)$. Thus, since $H$ is clearly not $0$, the result follows.
\end{proof}

As a consequence of Proposition~\ref{prop:density_intensity_SLE}, we have the local finiteness of $\mu^0$, as a measure on $\eta$.

\begin{lem}
Almost surely, $\mu^0$ is locally finite.  That is,
\begin{align*}
    \mu^0(\eta([s,t])) < \infty \quad\text{for all}\quad 0 < s < t \quad\text{a.s.}
\end{align*}
\end{lem}
\begin{proof}
Fix $0 < s < t$.  Then $\p[ \eta([s,t]) \subseteq B(0,R)] \to 1$ as $R \to \infty$.  On $\{\eta([s,t]) \subseteq B(0,R)\}$ we have that  $\mu^0(\eta([s,t])) \leq \mu^0(B(0,R))$.  Proposition~\ref{prop:density_intensity_SLE} implies that $\E[ \mu^0(B(0,R))] < \infty$.  In particular, we a.s.\ have that $\mu^0(B(0,R)) < \infty$ for all $R > 0$.  Therefore $\mu^0(\eta([s,t])) < \infty$ a.s.
\end{proof}

Thus since $\mu^0$ is locally finite and satisfies the correct conformal covariance formula and $(\eta^0,\mu^0) \stackrel{d}{=} (\eta^s,\wt \mu^s)$ holds, we have that $\mu^0$ is indeed the natural parameterization, and hence, (a constant times) the Minkowski content of $\SLE_{\kappa'}$ for $\kappa' \in (4,8)$.

\section{Existence of the $\CLE$ volume measure}
\label{sec:ccCLE}

In this section we construct a family $(\p_D)$ of probability measures on pairs $(\Gamma,\Xi)$ consisting of a $\CLE_\kappa$ in $D$ and a measure supported in its carpet ($\kappa \leq 4$) or gasket ($\kappa > 4$) which is conformally covariant with exponent $d=d(\kappa)$. That is, we shall prove the existence part of Theorem~\ref{thm:mainresult}. We use the natural LQG measures which are defined on standard and generalized quantum disks. However, by virtue of absolute continuity (Section~\ref{sec:LQG}), we may perform calculations with a zero-boundary GFF in place of the quantum disks.

\subsection{Simple conformal loop ensembles}\label{sec:ccsCLE}

Fix $\kappa \in (8/3,4)$ and let $D \subseteq \C$ be simply connected. We define the probability measure $\p_D$ to be the law on pairs $(\Gamma,\Xi)$ where $\Gamma \sim \p_D^{\CLE_\kappa}$ and $\Xi$ is defined as follows. Let $h$ be a zero-boundary GFF in $D$ which is independent of $\Gamma$, fix the parameters $\gamma = \sqrt{\kappa}$ and $\alpha = 4/\kappa$ and recall from Section~\ref{sec:nmCLE} the definition of $\Lambda^{\Gamma,h}$ (note that by absolute continuity, this makes sense with a zero-boundary GFF $h$ in place of the quantum disk). Moreover, when considering $\Lambda^{\Gamma,h}$, we let the point mass associated with the loop $\CL$ of the counting measure, $N_\epsilon^{\Gamma,h}$, lie on a point $m \in \CL$, sampled uniformly from the quantum length measure of the loop (renormalized to be a probability measure). Then $\Xi$ is the random measure defined by
\begin{align}\label{eq:measure}
    \Xi(dz) = \Xi^\kappa(dz) = F_D(z) \E\!\left[\Lambda^{\Gamma,h}(dz) \, \middle| \, \Gamma \right]\!,
\end{align}
where $F_D(z) = F_D^\kappa(z) = r_D(z)^{-(\frac{1}{2} + \frac{2}{\kappa} + \frac{\kappa}{32})}$ and $r_D(z)$ is the conformal radius of $D$ as seen from $z$. Moreover, we define $\Xi$ to be zero on $\partial D$. We begin by proving that $\E[\Xi]$ is locally finite.

\begin{lem}\label{lem:locally_finite_carpet}
Let $D \subseteq \C$ be simply connected and let $(\Gamma,\Xi) \sim \p_D$ where $\p_D$ is as above. Then $\E[\Xi]$ is locally finite. Consequently, $\Xi$ is a.s.\ locally finite.
\end{lem}
\begin{proof}
Fix $\kappa \in (8/3,4)$ and let $\gamma = \sqrt{\kappa}$. We begin by noting that it was proved in~\cite{msw2020simpleclelqg} that if $(D,h^D)$ is a unit boundary length $\gamma$-quantum disk and $\Gamma \sim \p_D^{\CLE_\kappa}$, then $\E\!\left[\Lambda^{\Gamma,h^D}(D)\right] = 1$. Moreover,  if we add a constant $C$ to the field, then the resulting carpet measure is multiplied by a factor of $e^{(\alpha + 1/2)C\frac{\gamma}{2}}$. This follows since adding a constant $C$ to a field $h^*$ scales quantum lengths by $e^{C\frac{\gamma}{2}}$ and hence for $U \subseteq D$ we have that
\begin{align}\label{eq:scaling_carpet_measure}
	\Lambda^{\Gamma,h^*+C}(U) &= \lim_{\epsilon \to 0} \epsilon^{\alpha+\frac{1}{2}} N_\epsilon^{\Gamma,h^*+C}(U) = \lim_{\epsilon \to 0} \epsilon^{\alpha+\frac{1}{2}} N_{\epsilon e^{-C\frac{\gamma}{2}}}^{\Gamma,h^*}(U) \\
	&= e^{(\alpha + \frac{1}{2})C \frac{\gamma}{2}} \lim_{\epsilon \to 0} (\epsilon e^{-C\frac{\gamma}{2}})^{\alpha+\frac{1}{2}} N_{\epsilon e^{-C\frac{\gamma}{2}}}^{\Gamma,h^*}(U) = e^{(\alpha + \frac{1}{2})C \frac{\gamma}{2}} \Lambda^{\Gamma,h^*}(U). \nonumber
\end{align}
Fix some compact $K \subseteq D$ and note that $\delta \coloneqq \dist(K,\partial D) > 0$.  Then we can decompose $h^D$ into $h^D = h + \Fh$ where $h$ is a zero-boundary GFF on $D$, $\Fh$ is a distribution which is harmonic on $D$, and $h$ and $\Fh$ are independent.  Since $\delta > 0$ it follows that $\Fh$ is a.s.\ bounded on $K$ and consequently that $\p[\sup_{z \in K} | \Fh(z) | \leq C] \to 1$ as $C \to \infty$. Thus if we let $E_K = \{\sup_{z \in K} | \Fh(z) | \leq C\}$ then we can choose $C > 0$ sufficiently large so that $\p[E_K] \geq 1/2$.  By the independence of $h$ and $\Fh$,~\eqref{eq:scaling_carpet_measure}, the inequality $r_D(z) \geq \dist(z,\partial D)$ and the fact that $\E[ \Lambda^{\Gamma,h^D}(D)] = 1$, we have that
\begin{align*}
	\E[\Xi](K) &= \int_K F_D(z) \E\!\left[ \Lambda^{\Gamma,h}(dz) \right] \leq 2 \int_K F_D(z) \E\!\left[ \one_{E_K} \Lambda^{\Gamma,h}(dz) \right] \\
	&\leq 2 e^{(\alpha + \frac{1}{2})C \frac{\gamma}{2}} \int_K F_D(z) \E\!\left[ \one_{E_K} \Lambda^{\Gamma,h^D}(dz) \right] \\
	&\leq 2e^{(\alpha + \frac{1}{2})C \frac{\gamma}{2}} \delta^{-(\frac{1}{2}+\frac{2}{\kappa}+\frac{\kappa}{32})} \E\!\left[ \Lambda^{\Gamma,h^D}(K) \right] \\
	&\leq 2e^{(\alpha + \frac{1}{2})C \frac{\gamma}{2}} \delta^{-(\frac{1}{2}+\frac{2}{\kappa}+\frac{\kappa}{32})} < \infty.
\end{align*}
Thus $\E[\Xi]$ is locally finite. The second assertion follows immediately from the first.
\end{proof}

\begin{lem}\label{lem:conformal_covariance_carpet}
The family of measures $(\p_D)$ defined above is conformally covariant with exponent $d = d(\kappa)$ given by~\eqref{eq:CLEdim}. Moreover, $(\p_D)$ respects the $\CLE_\kappa$ Markov property.
\end{lem}
\begin{proof}
We begin by proving the conformal covariance.  Suppose that $(\Gamma,\Xi) \sim \p_D$ and let $h$ and $\Lambda^{\Gamma,h}$ be as above. Consider a simply connected domain $\wt{D} \subseteq \C$ and a conformal map $\varphi: D \to \wt{D}$ and a zero-boundary GFF $\wt{h}$ on $\wt{D}$ which is independent of $\Gamma$ and $h$ (the latter independence does not matter). Then $\wt{\Gamma} = \varphi(\Gamma) \sim \p_{\wt{D}}^{\CLE_\kappa}$ and $h^0 = \wt{h} \circ \varphi$ is a zero-boundary GFF in $D$ independent of $\Gamma$.  We let $\wt{\Xi}$ be defined by~\eqref{eq:measure} with $\wt{\Gamma}$ and $\wt{h}$ in place of $\Gamma$ and $h$, respectively. Then we have as in~\eqref{eq:scaling_carpet_measure} that
\begin{align*}
	\wt{\Xi} \circ \varphi(dz) &= F_{\wt{D}}(\varphi(z)) \E\!\left[ \Lambda^{\wt{\Gamma},\wt{h}} \circ \varphi(dz) \, \middle| \, \wt{\Gamma} \right] \\
	&= F_{\wt{D}}(\varphi(z)) \E\!\left[ \lim_{\epsilon \rightarrow 0} \epsilon^{\alpha+\frac{1}{2}} N_\epsilon^{\wt{\Gamma},\wt{h}} \circ \varphi(dz) \, \middle| \, \wt{\Gamma} \right] \\
	&= F_{\wt{D}}(\varphi(z)) \E\!\left[ \lim_{\epsilon \rightarrow 0} \epsilon^{\alpha+\frac{1}{2}} N_\epsilon^{\Gamma,\wt{h}\circ \varphi + Q \log |\varphi'| }(dz) \, \middle| \, \wt{\Gamma} \right] \\
	&= F_{\wt{D}}(\varphi(z)) \E\!\left[ \lim_{\epsilon \rightarrow 0} \epsilon^{\alpha+\frac{1}{2}} N_{\epsilon |\varphi'|^{-Q \frac{\gamma}{2}}}^{\Gamma,h^0}(dz) \, \middle| \, \Gamma \right] \\
	&= F_{\wt{D}}(\varphi(z)) \E\!\left[ |\varphi'(z)|^{(\alpha + \frac{1}{2}) Q \frac{\gamma}{2}} \lim_{\epsilon \rightarrow 0} \left(\epsilon |\varphi'(z)|^{-Q \frac{\gamma}{2}}\right)^{\alpha+\frac{1}{2}} N_{\epsilon |\varphi'|^{-Q \frac{\gamma}{2}}}^{\Gamma,h^0}(dz) \, \middle| \, \Gamma \right] \\
	&= F_{\wt{D}}(\varphi(z)) |\varphi'(z)|^{(\alpha + \frac{1}{2}) Q \frac{\gamma}{2}} \E\!\left[ \Lambda^{\Gamma,h^0}(dz) \, \middle| \, \Gamma \right]\!.
\end{align*}

Finally, noting that
\begin{align*}
	F_{\wt{D}}(\varphi(z)) = |\varphi'(z)|^{-(\frac{1}{2} + \frac{2}{\kappa} + \frac{\kappa}{32})} F_D(z)
\end{align*}
and that $(\alpha + \tfrac{1}{2})Q \tfrac{\gamma}{2} -(\tfrac{1}{2} + \tfrac{2}{\kappa} + \tfrac{\kappa}{32}) = d$, we have that
\begin{align*}
	\wt{\Xi} \circ \varphi(dz) = |\varphi'(z)|^d \Xi(dz),
\end{align*}
which proves the conformal covariance.

We now turn to the proof that $(\p_D)$ respects the $\CLE_\kappa$ domain Markov property. Consider $U \subseteq D$ simply connected, let $V$ be a component of $U^\Gamma$ (recall Section~\ref{sec:intro}) and let $\Gamma_V$ denote the loops of $\Gamma$ which are contained in $V$.  Let $\phi: \D \to V$ be a conformal map and write $\wh{\Gamma} = \phi^{-1}(\Gamma_V) \sim \p_{\D}^{\CLE_\kappa}$.

As in the proof of Lemma~\ref{lem:measure_SLE}, we let $\FH$ be a Gaussian field on $V$ with $\textup{Var} \, \FH(z) = \log r_{\D}(z) - \log r_{\D}(\phi^{-1}(z)) + \log|(\phi^{-1})'(z)|$ taken to be independent of $\Gamma_V$ and $h$.  We then set $\wh{h} = h \circ \phi - \FH \circ \phi$ so that $\wh{h}$ is a zero-boundary GFF on $\D$ independent of $\wh{\Gamma}$.  Thus
\begin{align*}
    \Xi|_V \circ \phi(dz) &= F_{\D}(\phi(z)) \E\!\left[ \Lambda^{\Gamma,h} \circ \phi(dz) \, \middle| \, \Gamma \right] \\
    &= F_{\D}(\phi(z)) \E\!\left[ \lim_{\epsilon \rightarrow 0} \epsilon^{\alpha+\frac{1}{2}} N_\epsilon^{\Gamma,h} \circ \phi(dz) \, \middle| \, \Gamma \right] \\
    &= F_{\D}(\phi(z)) \E\!\left[ \lim_{\epsilon \rightarrow 0} \epsilon^{\alpha+\frac{1}{2}} N_\epsilon^{\wh{\Gamma},h\circ \phi + Q \log |\phi'|}(dz) \, \middle| \, \Gamma \right] \\
    &= F_{\D}(\phi(z)) \E\!\left[ \lim_{\epsilon \rightarrow 0} \epsilon^{\alpha+\frac{1}{2}} N_{\epsilon|\phi'|^{-Q\frac{\gamma}{2}}}^{\wh{\Gamma},h\circ \phi}(dz) \, \middle| \, \Gamma \right] \\
    &= F_{\D}(\phi(z)) \E\!\left[ \lim_{\epsilon \rightarrow 0} \epsilon^{\alpha+\frac{1}{2}} N_{\epsilon|\phi'|^{-Q\frac{\gamma}{2}}e^{-\frac{\gamma}{2} \FH \circ \phi}}^{\wh{\Gamma},\wh{h}}(dz) \, \middle| \, \Gamma \right] \\
    &= F_{\D}(\phi(z)) \E\!\left[ |\phi'(z)|^{(\alpha+\frac{1}{2})Q\frac{\gamma}{2}}e^{(\alpha+\frac{1}{2})\frac{\gamma}{2} \FH \circ \phi(z)} \right.\\
    &\qquad \qquad \qquad\left. \times \lim_{\epsilon \rightarrow 0} (\epsilon|\phi'(z)|^{-Q\frac{\gamma}{2}}e^{-\frac{\gamma}{2} \FH \circ \phi(z)})^{\alpha+\frac{1}{2}} N_{\epsilon|\phi'|^{-Q\frac{\gamma}{2}}e^{-\frac{\gamma}{2} \FH \circ \phi}}^{\wh{\Gamma},\wh{h}}(dz) \, \middle| \, \Gamma \right] \\
    &= F_{\D}(\phi(z)) |\phi'(z)|^{(\alpha+\frac{1}{2})Q\frac{\gamma}{2}} \E\!\left[ e^{(\alpha+\frac{1}{2})\frac{\gamma}{2} \FH \circ \phi(z)} \, \middle| \, \Gamma \right] \E\!\left[ \Lambda^{\wh{\Gamma},\wh{h}}(dz) \, \middle| \, \wh{\Gamma} \right].
\end{align*}
Next, noting that since $\frac{1}{2}(\alpha+\frac{1}{2})^2 \frac{\gamma^2}{4} = \frac{1}{2} + \frac{2}{\kappa} + \frac{\kappa}{32}$, we have
\begin{align*}
    \E\!\left[ e^{(\alpha+\frac{1}{2})\frac{\gamma}{2} \FH \circ \phi(z)} \, \middle| \, \Gamma \right] &= r_{\D}(\phi(z))^{\frac{1}{2} + \frac{2}{\kappa} + \frac{\kappa}{32}}  r_{\D}(z)^{-(\frac{1}{2} + \frac{2}{\kappa} + \frac{\kappa}{32})} |\phi'(z)|^{-(\frac{1}{2} + \frac{2}{\kappa} + \frac{\kappa}{32})} \\
    &= \frac{F_{\D}(z)}{F_{\D}(\phi(z))} |\phi'(z)|^{-(\frac{1}{2} + \frac{2}{\kappa} + \frac{\kappa}{32})}.
\end{align*}
Thus, again since $(\alpha+\frac{1}{2})Q\frac{\gamma}{2} - (\frac{1}{2} + \frac{2}{\kappa} + \frac{\kappa}{32}) = d$, we have that 
\begin{align*}
    \Xi|_{V} \circ \phi(dz) = |\phi'(z)|^d \wh{\Xi}(dz),
\end{align*}
where $\wh{\Xi}$ is such that $(\wh{\Gamma},\wh{\Xi}) \sim \p_{\D}$. Consequently it follows that the conditional law of $(\Gamma_V, \Xi_V)$ given $\Gamma \setminus \Gamma_V$ and $\Xi|_{D \setminus V}$ is $\p_V$.
\end{proof}

\subsection{Non-simple conformal loop ensembles}
Fix $\kappa' \in (4,8)$, let $\gamma = 4/\sqrt{\kappa'}$ and $\alpha' = 4/\kappa'$ and define the law $\p_D$ on pairs $(\Gamma,\Xi)$ where the marginal law of $\Gamma$ is that of a $\CLE_{\kappa'}$ in $D$ and $\Xi$ is defined as follows. Let $h$ be a zero-boundary GFF in $D$ and let $\Lambda^{\Gamma,h}$ be the natural LQG measure on the $\CLE_{\kappa'}$ gasket, defined in Section~\ref{sec:nmCLE}, with $h$ in place of the generalized quantum disk distribution.  We then define $\Xi$ by
\begin{align*}
	\Xi(dz) = \Xi^{\kappa'}(dz) = F_D(z) \E\!\left[ \Lambda^{\Gamma,h}(dz) \, \middle| \, \Gamma \right],
\end{align*}
where $F_D(z) = F_D^{\kappa'}(z) = r_D(z)^{-(\frac{1}{2} + \frac{2}{\kappa'} + \frac{\kappa'}{32})}$. Again, as in the case of the simple $\CLE$s, we let the measure~$\Xi$ be zero on $\partial D$.

The key in the proof of conformal covariance for the measures on simple $\CLE$ was the way quantum length scales when adding a constant to the field.  That is, when we add $A$ to the field, quantum length is scaled by $e^{\frac{\sqrt{\kappa}}{2} A}$ so that
\[ N_\epsilon^{\Gamma,h+A} = N_{\epsilon e^{-\frac{\sqrt{\kappa}}{2} A}}^{\Gamma,h}.\]
Hence, if we show that in the case of $\kappa' \in (4,8)$ that
\[ N_\epsilon^{\Gamma,h+A} = N_{\epsilon e^{-\frac{\sqrt{\kappa'}}{2} A}}^{\Gamma,h},\]
 then the conformal covariance follows from the same calculations. Below, we consider $\CLE_{\kappa'}$ drawn on a generalized quantum disk as the description on this surface is precise. That the same scaling rules hold when replacing the distribution by a zero-boundary GFF follows by absolute continuity.

Under the quantum natural time parameterization, the left/right boundary lengths of an $\SLE_{\kappa'}$-type curve that make up a $\CLE_{\kappa'}$ loop are absolutely continuous with respect to independent totally asymmetric $\frac{\kappa'}{4}$-stable processes with only negative jumps (see~\cite{dms2014mating}). The generalized quantum length of the loop is then defined as the (normalized) limit of the number loops with boundary length in $[\epsilon,2\epsilon]$ as $\epsilon \rightarrow 0$, that is, the (normalized) limit of the number of negative jumps of the two $\frac{\kappa'}{4}$-stable processes which have size in $[\epsilon,2\epsilon]$.

The law of the number of negative jumps of size belonging to a set $I \subseteq (0,\infty)$ by time $t$ is that of a Poisson process with jump intensity given by $\Pi(I)$ where $\Pi(dy) = C y^{-\frac{\kappa'}{4}-1}dy$ is the L\'evy measure of the stable process (see~\cite{bertoin1996book}) and $C > 0$ is some constant. Thus, if $N_\epsilon$ denotes the number of jumps of a $\frac{\kappa'}{4}$-stable process with size in $[\epsilon,2\epsilon]$ then $\epsilon^{\frac{\kappa'}{4}} N_\epsilon$ converges a.s.\ to a constant as $\epsilon \rightarrow 0$.

Thus, considering a single $\SLE_{\kappa'}$-type loop $\CL$ in the $\CLE_{\kappa'}$ the generalized quantum length changes as follows when we add a constant $A$ to the field. The quantum boundary lengths of the quantum disks surrounded by $\CL$ are multiplied by $e^{A\frac{\gamma}{2}}$. Hence, the quantum length of the loop $\CL$ as measured by the field (recall Remark~\ref{rmk:qlengthCLE}) is given by
\begin{align*}
    \sigma^{\CL,h+A}(\CL) = \lim_{\epsilon \rightarrow 0} \epsilon^{\frac{\kappa'}{4}} N_\epsilon^{\CL,h+A}(\CL) = \lim_{\epsilon \rightarrow 0} \epsilon^{\frac{\kappa'}{4}} N_{\epsilon e^{-\frac{\gamma}{2} A}}^{\CL,h}(\CL) = e^{\frac{\kappa'}{4}\frac{\gamma}{2} A} \sigma^{\CL,h}(\CL)  = e^{\frac{\sqrt{\kappa'}}{2} A} \sigma^{\CL,h}(\CL).
\end{align*}
Consequently, we have that
\begin{align*}
    N_\epsilon^{\Gamma,h+A} = N_{\epsilon e^{\frac{\sqrt{\kappa'}}{2} A}}^{\Gamma,h},
\end{align*}
which implies for $U \subseteq D$ that
\begin{align}
\label{eq:scaling_gasket_measure}
	\Lambda^{\Gamma,h+A}(U) = e^{(\alpha' + \frac{1}{2})A \frac{\sqrt{\kappa'}}{2}} \Lambda^{\Gamma,h}.
\end{align}
Thus, if we prove the almost sure local finiteness of $\E[\Xi]$ then it follows by the exact same proof as that of Lemma~\ref{lem:conformal_covariance_carpet} that $(\p_D)$ satisfies the properties of Definition~\ref{def:cov_measure} with exponent $d = d(\kappa')$. 

\begin{lem}\label{lem:locally_finite_gasket}
Let $D \subseteq \C$ be simply connected and let $(\Gamma,\Xi) \sim \p_D$ where $\p_D$ is as above. Then $\E[\Xi]$ is locally finite. Consequently, $\Xi$ is a.s.\ locally finite.
\end{lem}
\begin{proof}
Fix $\kappa' \in (4,8)$ and let $\gamma = 4/\sqrt{\kappa'}$. It was proved in~\cite{msw2020nonsimpleclelqg} that if $(D,h^{GD})$ is a unit boundary length $\gamma$-generalized quantum disk and $\Gamma \sim \p_D^{\CLE_{\kappa'}}$, then $\E\!\left[ \Lambda^{\Gamma,h^{GD}}(D) \right] = 1$.  Recall that in order to construct a generalized quantum disk, one first samples a $\tfrac{4}{\gamma^2}$-stable loop tree, where the loop lengths are given by the sizes of the jumps a $\tfrac{4}{\gamma^2}$-stable L\'{e}vy excursion. One then samples ordinary quantum disks in the loops, with boundary length given by the length of the loops. In a unit boundary length generalized quantum disk, the stable L\'{e}vy excursion runs for time $1$ and hence has a positive probability of having a downward jump with size in $[1,2]$.  Let $E$ be this event and let $(D,h^D)$ be a unit boundary length $\gamma$-quantum disk.  Then we have that
\[ 1 = \E\!\left[ \Lambda^{\Gamma,h^{GD}}(D) \right] \geq \E\!\left[ \Lambda^{\Gamma,h^{GD}}(D) \one_E \right] \geq \p[E] \E[\Lambda^{\Gamma,h^D}(D)].\]
In the final inequality, we used scaling to conclude that $\E[\Lambda^{\Gamma,h^D}(D)]$ would be larger if we replaced $(D,h^D)$ by a quantum disk with boundary length $\ell \in [1,2]$.  As we mentioned before, $\p[E] > 0$ so we conclude that $\E[\Lambda^{\Gamma,h^D}(D)] < \infty$.

Moreover, considering a compact set $K \subseteq D$ as before and writing $\delta \coloneqq \dist(K,\partial D) > 0$, as in the proof of Lemma~\ref{lem:locally_finite_carpet} we write $h^D = h + \Fh$, where $h$ is a zero-boundary GFF on $D$, and $\Fh$ harmonic and independent of $h$.  Let $C$ be a constant such that $\p[ \sup_{z \in K} | \Fh(z) | \leq C] \geq 1/2$.  Then
\begin{align*}
	\Lambda^{\Gamma,h}(K) \leq 2 e^{(\alpha' + \frac{1}{2}) C \frac{\sqrt{\kappa'}}{2}} \Lambda^{\Gamma,h^D}(K) \quad\text{on}\quad \sup_{z \in K} | \Fh(z) | \leq C.
\end{align*} 
Hence we have that
\begin{align*}
	\E[\Xi](K) &= \int_K F_D(z) \E[\Lambda^{\Gamma,h}(dz)] \leq \delta^{-(\frac{1}{2} + \frac{2}{\kappa'} + \frac{\kappa'}{32})} \E[\Lambda^{\Gamma,h}(K)] \\
	&= 2 e^{(\alpha' + \frac{1}{2}) C \frac{\sqrt{\kappa'}}{2}} \delta^{-(\frac{1}{2} + \frac{2}{\kappa'} + \frac{\kappa'}{32})} \E[\Lambda^{\Gamma,h^D}(K)] < \infty.
\end{align*}
Thus $\E[\Xi]$ is locally finite and the second assertion follows immediately.
\end{proof}

By Lemma~\ref{lem:locally_finite_gasket} and the discussion before it, we have that the next lemma follows by repeating the proof of Lemma~\ref{lem:conformal_covariance_carpet}.

\begin{lem}\label{lem:conformal_covariance_gasket}
The family of measures $(\p_D)$ defined above is conformally covariant with exponent $d = d(\kappa')$ given by~\eqref{eq:CLEdim}. Moreover, $(\p_D)$ respects the $\CLE_\kappa$ Markov property.
\end{lem}

\subsection{Existence of the measure on the $\CLE_4$ carpet}

Since the natural LQG measure on the $\CLE_\kappa$ carpet was not previously constructed in the critical case $\gamma = 2$ ($\kappa=4$), we will instead consider the weak limit of the measures defined above as $\kappa \nearrow 4$ in order to construct a conformally covariant measure on the $\CLE_4$ carpet. This is done by coupling the $\CLE$ in the way given by the loop-soup construction, see Section~\ref{sec:loop-soup}. We will prove the following.

\begin{prop}
\label{prop:CLE4_measure}
There exists a law $(\p_D) = (\p_D^4)$ as in Definition~\ref{def:cov_measure} on pairs $(\Gamma,\Xi)$ where the marginal law of $\Gamma$ is that of a $\CLE_4$ in $D$ with conformal covariance exponent $d=d(4)= 15/8$.
\end{prop}
\begin{proof}
For each $\kappa \in (8/3,4)$ and $D \subseteq \C$ simply connected let $\p_D^\kappa$ be the law on pairs $(\Gamma^\kappa,\Xi^\kappa)$ where the marginal law of $\Gamma^\kappa$ is that of a $\CLE_\kappa$ on $D$ and $\Xi^\kappa$ is as in Section~\ref{sec:ccsCLE}.  We assume that $\Xi^\kappa$ is normalized to have expected total mass $1$ in the case that $D = \D$ (which in turn serves to normalize the expected mass on general $D \subseteq \C$ simply connected by conformal covariance).

Pick a sequence $\kappa_n \nearrow 4$ (equivalently, a sequence of intensity constants $c_n \nearrow 1$ of the Brownian loop-soups).  For each $n \in \N$ we let $(\Gamma^{\kappa_n},\Xi^{\kappa_n}) \sim \p_\D^{\kappa_n}$.  By coupling the $\CLE_{\kappa_n}$ as mentioned in Section~\ref{sec:loop-soup} we have that the sequence of $\CLE_{\kappa_n}$ carpets converges to the $\CLE_4$ carpet in the topology defined by the Hausdorff distance. This holds by monotonicity since the $\CLE_{\kappa_n}$ carpets form a decreasing sequence of sets (since the loop-soups defining them increase with $n$).  By~\cite[Theorem~4.10]{kallenberg2017rmbook}, the collection of measures $(\Xi^{\kappa_n})$ is weakly tight. It follows that $\Xi^{\kappa_n}$ converges weakly in law (at least along a subsequence, and if so, consider this subsequence instead) to a random measure $\Xi^4$ as $n \rightarrow \infty$. That the support of $\Xi^4$ is contained in the $\CLE_4$ carpet follows since the support of $\Xi^{\kappa_n}$ is contained in the $\CLE_{\kappa_n}$ carpet, which converges to the $\CLE_4$ carpet $\Gamma^4$ in $D$ as $\kappa_n \nearrow 4$.  Applying the Skorokhod representation theorem for weak convergence, we can (after possibly recoupling the laws onto a new probability space) assume that $(\Gamma^{\kappa_n},\Xi^{\kappa_n}) \to (\Gamma^4, \Xi^4)$ a.s.\ where the notion of convergence in the first (resp.\ second) coordinate is the Hausdorff topology (resp.\ weak convergence).

Suppose that $\wt{D} \subseteq \C$ is a simply connected domain and let $\varphi : \D \to \wt{D}$ be a conformal transformation.  For each $n \in \N$ we let $(\wt{\Gamma}^{\kappa_n}, \wt{\Xi}^{\kappa_n}) \sim \p_{\wt{D}}^{\kappa_n}$ be such that $\wt{\Gamma}^{\kappa_n} = \varphi(\Gamma^{\kappa_n})$ and $\wt{\Xi}^{\kappa_n}$ is defined for each Borel set $A$ by
\[ \wt{\Xi}^{\kappa_n}(A) = \int_{\varphi^{-1}(A)} |\varphi'(z)|^{d_n} d\Xi^{\kappa_n}(z)\]
where $d_n = d(\kappa_n)$.  We claim that $\wt{\Xi}^{\kappa_n}$ restricted to any compact set in $\wt{D}$ converges as $n \to \infty$ weakly in probability  to the measure $\wt{\Xi}^4$ defined by 
\[ \wt{\Xi}^4(A) = \int_{\varphi^{-1}(A)} |\varphi'(z)|^{d} d\Xi^4(z).\]
To see this, note that (recalling the notation $\Xi^\kappa(f) = \int f d\Xi^\kappa$) if $f$ is any continuous function which is compactly supported in $\D$ then 
\begin{align*}
    \Xi^{\kappa_n}(f |\varphi'|^{d_n}) - \Xi^{4}(f |\varphi'|^{d}) = \Xi^{\kappa_n}(f (|\varphi'|^{d_n}-|\varphi'|^{d})) - (\Xi^{4}-\Xi^{\kappa_n})(f |\varphi'|^{d}).
\end{align*}
Let $K \subseteq \D$ be compact so that the support of $f$ is contained in $K$.  Then we have that
\[ \E[ |\Xi^{\kappa_n}(f (|\varphi'|^{d_n}-|\varphi'|^{d}))|] \leq \sup_{z \in K} |f(z) (|\varphi'(z)|^{d_n}-|\varphi'(z)|^{d})| \to 0 \quad\text{as}\quad n \to \infty.\]
We also have that $(\Xi^{4}-\Xi^{\kappa_n})(f |\varphi'|^{d}) \to 0$ a.s.\ since $\Xi^{\kappa_n} \to \Xi^4$ a.s.\ weakly as $n \to \infty$.  Altogether, this proves the claim.

We have thus constructed a family of measures $(\p_D^4)$ on pairs consisting of a $\CLE_4$ and a Borel measure supported on the $\CLE_4$ carpet which is conformally covariant.  It may not be that the Borel measure we have constructed is determined by the $\CLE_4$ carpet but by replacing it with its conditional expectation given the $\CLE_4$ carpet we may assume that it is the case.  To finish the proof, it is left to show that the family of measures $(\p_D^4)$ respects the $\CLE_4$ Markov property.  To prove that this is the case, it suffices to verify it in the case $D = \D$ by conformal covariance.  Fix $U \subseteq \D$ simply connected.  Assume that we have $(\Gamma^{\kappa_n},\Xi^{\kappa_n}) \sim \p_D^{\kappa_n}$ as above where the $\Gamma^{\kappa_n}$ are coupled together using a common instance of the Brownian loop soup (so their carpets decrease to a $\CLE_4$ carpet).  For each $n \in \N$, we let $U_n^\Gamma$ the set of points obtained $U$ by removing from $U$ the closure of the union of the set of points surrounded by loops of $\Gamma^{\kappa_n}$ which intersect both $U$ and $\D \setminus U$.  Let $U^\Gamma$ be the set of points obtained by removing from $U$ the closure of the union of the points surrounded by loops of the limiting $\CLE_4$ which intersect both $U$ and $D \setminus U$.  Fix $z \in U$ deterministic and let $V$ be the component of $U^\Gamma$ which contains $z$ (we note that it could be that $V = \emptyset$).  For each $n \in \N$, let~$V_n$ be the component of~$U_n^\Gamma$ which contains~$z$.  Then we note that the~$V_n$ are decreasing.

For each $n \in \N$, let $\varphi \colon \D \to V_n$ be the unique conformal transformation with $\varphi_n(0) = z$ and $\varphi_n'(0) > 0$.  Then the sequence $(\varphi_n)$ converges locally uniformly to the unique conformal map $\varphi \colon \D \to V$ with $\varphi(0) = z$ and $\varphi'(0) > 0$.  For each $n \in \N$, let $\F_n$ be the $\sigma$-algebra generated by the Brownian loops used to generate $\Gamma^{\kappa_n}$ which are part of a cluster corresponding to a loop of $\Gamma^{\kappa_n}$ which intersects $U$ and $\D \setminus U$.  Then we note that the $\sigma$-algebra generated by all of the $\F_n$'s is equal to that generated by the Brownian loops used to construct $\Gamma^4$ which are part of a cluster corresponding to a loop of $\Gamma^4$ which intersects $U$ and $D \setminus U$.  Since the measures $\p_D^{\kappa_n}$ respect the $\CLE_{\kappa_n}$ Markov property for each $n$, it follows that the conditional law of the restriction of the pair consisting of the loops $\Gamma_{V_n}^{\kappa_n}$ of $\Gamma^{\kappa_n}$ contained in $V_n$ and the restriction of $\Xi^{\kappa_n}$ to $V_n$ given $\F_n$ is that of $\p_{V_n}^{\kappa_n}$.  Equivalently, the conditional law given $\F_n$ is equal to that of the image under $\varphi_n$ of a $\CLE_{\kappa_n}$ in $\D$ independent of $\F_n$ and the covariant pushforward of its measure under $\varphi_n$.  The claim thus follows from the arguments given above and the convergence of $\varphi_n$ to $\varphi$. 
\end{proof}

\section{Uniqueness of the $\CLE$ volume measure}
\label{sec:properties}

Since the existence of the conformally covariant measure was proved in Section~\ref{sec:ccCLE}, the part that we have left to establish is the uniqueness.  This is done by showing that the conformal covariance of the measure completely determines its intensity (Section~\ref{subsec:intensity_form}) and then combining this with an exploration of the $\CLE$ together with the property of respecting the $\CLE$ Markov property (Section~\ref{subsec:CLE_unique}). We now work on the unit disk $\D$.

\subsection{Form of the intensity}
\label{subsec:intensity_form}

Consider a sample $(\Gamma,\Xi)$ from a measure $\p_{\D}$ satisfying Definition~\ref{def:cov_measure}. Note that by part~\eqref{it:conf_cov} of Definition~\ref{def:cov_measure}, for any M\"{o}bius transformation $\phi:\D \to \D$,  letting $\wt{\Gamma} = \phi(\Gamma)$ and
\begin{align}
\label{eq:confcov}
	\wt{\Xi}(dz) = |\phi'(z)|^{-d} \Xi \circ \phi(dz),
\end{align}
it follows that $(\Gamma,\Xi) \stackrel{d}{=} (\wt{\Gamma},\wt{\Xi})$. Taking an expectation of both sides of~\eqref{eq:confcov} it follows that the intensity $\E[\Xi]$ of the measure $\Xi$ satisfies
\begin{align}\label{eq:intensity_cc}
    \E[\Xi](dz) = | \phi'(z) |^{-d} \E[\Xi] \circ \phi(dz)
\end{align}
for any M\"{o}bius transformation $\phi: \D \to \D$. Moreover, by part~\eqref{it:locally_finite_expectation} of Definition~\ref{def:cov_measure} we have that $\E[\Xi]$ is locally finite. 

We now prove that~\eqref{eq:intensity_cc} completely determines the form of $\E[\Xi]$. For any $D \subseteq \C$, let $\SB(D)$ denote the $\sigma$-algebra of Borel subsets of $D$.

\begin{lem}
\label{lem:density}
Let $m$ be a locally finite measure on the measurable space $(\D,\SB(\D))$ for which there exists $d \in \R$ such that for each M\"{o}bius transformation $\phi: \D \rightarrow \D$ we have
\begin{align}\label{eq:conformally_covariant_measure}
    m(dz) = |\phi'(z)|^{-d} m \circ \phi(dz).
\end{align}
Then there exists some constant $C_m$ such that
\begin{align*}
    m(dz) = C_m \!\left( \frac{1}{1-|z|^2} \right)^{2-d} dz.
\end{align*}
\end{lem}
\begin{proof}
We begin by showing that $m$ is absolutely continuous with respect to Lebesgue measure. It follows from the local finiteness that $m$ is $\sigma$-finite.   This allows us to apply the Lebesgue decomposition theorem to see that we can uniquely write
\begin{equation}
\label{eqn:leb_decomp}
dm(z) = f(z) dz + d\sigma(z),
\end{equation}
for some measurable function $f$ which is integrable on each compact subset of $\D$ and some measure $\sigma$ which is singular with respect to Lebesgue measure. We shall show that $\sigma = 0$. Since $m$ satisfies~\eqref{eq:conformally_covariant_measure}, the uniqueness of the decomposition~\eqref{eqn:leb_decomp} implies that $\sigma$ does as well.  For each $z \in \D$ we let
\begin{align*}
    \phi_z(w) = \frac{w-z}{1-\overline{z} w},
\end{align*}
so that $\phi_z \colon \D \to \D$ is the unique M\"obius transformation with $\phi_z(z) = 0$ and $\phi_z'(z) = (1-|z|^2)^{-1} > 0$.  Using the identity $\sigma(dw) = |\phi_z'(w)|^{-d} (\sigma \circ \phi)(dw)$ for all $z \in \D$ we thus have that
\[ \sigma(dw) = \frac{1}{\pi}\int_\D |\phi_z'(w)|^{-d} d(\sigma \circ \phi_z)(w) dz, \]
where here we emphasize that $dz$ denotes integration with respect to Lebesgue measure.
Suppose that $A \subseteq \D$ has zero Lebesgue measure.  Note that $\phi_z^{-1}(w) = \phi_{-z}(w)$  Then we have that
\[ \sigma(A) = \frac{1}{\pi} \int \int_\D \one_A(w) |\phi_z'(w)|^{-d} dz d(\sigma \circ \phi_z)(w) = \frac{1}{\pi} \int \left(\int_\D \one_A(\phi_z^{-1}(w)) |(\phi_z^{-1})'(w)|^d dz \right) d\sigma (w) = 0.\]
This implies that $\sigma$ is absolutely continuous with respect to Lebesgue measure hence $\sigma = 0$.

By assumption, we have for each conformal map $\phi : \D \to \D$ that
\begin{align}\label{eq:RNd1}
    \frac{d(m \circ \phi)}{dm}(z) = |\phi'(z)|^d.
\end{align}
Let $f$ be the density of $m$ with respect to Lebesgue measure.  Then we also have that
\begin{align*}
    \frac{d(m \circ \phi)}{dz} (z) = f(\phi(z)) |\phi'(z)|^2
\end{align*}
and thus
\begin{align}\label{eq:RNd2}
    \frac{d(m \circ \phi)}{dm}(z) = \frac{\frac{d(m \circ \phi)}{dz}(z)}{\frac{dm}{dz}(z)} = \frac{f(\phi(z)) |\phi'(z)|^2}{f(z)}.
\end{align}
Combining~\eqref{eq:RNd1} and~\eqref{eq:RNd2} we have
\begin{align}
\label{eqn:f_form}
    f(z) = f(\phi(z)) |\phi'(z)|^{2-d}.
\end{align}
Applying~\eqref{eqn:f_form} for the choice $\phi = \phi_z$ gives that
\begin{align*}
    f(z) = f(0)|\phi_z'(z)|^{2-d} = f(0) \!\left( \frac{1}{1-|z|^2} \right)^{2-d},
\end{align*}
which completes the proof.
\end{proof}

\subsection{Proof of uniqueness}
\label{subsec:CLE_unique}

This subsection is dedicated to the proof of the following uniqueness result.
\begin{prop}\label{prop:CLE_unique}
Fix $\kappa \in (8/3,8)$. Let $(\p_D)$ be the family of measures constructed in Section~\ref{sec:ccCLE} with $d = d(\kappa)$ given by~\eqref{eq:CLEdim}. Assume that $(\wt{\p}_D)$ is another family of probability measures satisfying Definition~\ref{def:cov_measure} with exponent $d$. Then there exists a deterministic constant $K>0$ such that if $(\Gamma,\Xi) \sim \p_D$ then $(\Gamma,K\Xi) \sim \wt{\p}_D$.
\end{prop}

The main input in the proof of Proposition~\ref{prop:CLE_unique} is Lemma~\ref{lem:density}, which together with a suitable exploration of the $\CLE$ gives us the form of the conditional expectation of the measures on the unexplored parts. First, however, we shall prove the following lemma.

\begin{lem}\label{lem:condint}
Consider a probability space $(\Omega, \F, \p)$, let $\G \subseteq \F$ be a $\sigma$-algebra, let $g$ be a non-negative $\G \times \SB(\D)$-measurable function, and let $\M$ be a random measure on $\D$.  Then
\begin{align}\label{eq:conditional_intensity}
    \E\!\left[ \int_{\D} g(z) d\M(z) \, \middle| \, \G \right] = \int_{\D} g(z) d\E[\M \, | \, \G](z) \quad\text{a.s.}
\end{align}
Consequently, if $\M$ is independent of $\G$ then
\begin{align*}
    \E\!\left[ \int_{\D} g(z) d\M(z) \, \middle| \, \G \right] = \int_{\D} g(z) d\E[\M](z) \quad\text{a.s.}
\end{align*}
\end{lem}
\begin{proof}
We note that the second assertion follows immediately from the first. Moreover, we are done if we prove for all $A \in \G$ that
\begin{align}
\label{eq:conditional_intensity_sufficient}
	\E\!\left[ \one_A \int_{\D} g(z) d\M(z)\right] = \E\!\left[ \one_A \int_{\D} g(z) d\E[\M \, | \, \G](z) \right]\!.
\end{align}
Let $\CG$ be the set of bounded $\G \times \SB(\D)$-measurable functions such that~\eqref{eq:conditional_intensity_sufficient} holds for all $A \in \G$.  Suppose that $B \times O \in \G \times \SB(\D)$ and we let $g = \one_{B \times O}$.  Then the left-hand side of~\eqref{eq:conditional_intensity_sufficient} is equal to $\E[ \one_{A \cap B} \M(O)]$ and the right-hand side is equal to $\E[ \one_{A \cap B} \E[\M \, | \, \G] (O) ]$. Since $A \cap B \in \G$, it follows from the definition of conditional intensity that these quantities are equal. Moreover, by the linearity of the expectation and the integral, it follows that~$\CG$ is a vector space. Furthermore, we note that if~$g_n$ is a sequence of non-negative functions in $\CG$ which increase to a bounded function $g$ then $g \in \CG$ by the monotone convergence theorem. Consequently, by the monotone class theorem, $\CG$ contains all bounded non-negative $\G \times \SB(\D)$-measurable functions.  Finally we note that if $g$ is an arbitrary non-negative $\G \times \SB(\D)$-measurable function, then we can pick a sequence of functions $g_n \in \CG$ which increase to~$g$ and by applying the monotone convergence theorem once more, this implies that $g$ satisfies~\eqref{eq:conditional_intensity_sufficient}. Consequently,~\eqref{eq:conditional_intensity_sufficient} and hence,~\eqref{eq:conditional_intensity} hold for all non-negative $\G \times \SB(\D)$-measurable functions.
\end{proof}

\begin{lem}
\label{lem:loop_has_zero_mass}
Suppose that $(\Gamma,\Xi) \sim \p_\D$ where $\p_D$ is as in the statement of Proposition~\ref{prop:CLE_unique}.  Then we a.s.\ have for each loop $\CL \in \Gamma$ that $\Xi(\CL) = 0$.
\end{lem}
\begin{proof}
This holds for the quantum measure hence holds for $\Xi$ by construction for $\kappa \neq 4$.  Therefore we just have to give the proof in the case that $\kappa = 4$.  It suffices to show for each $z \in \D$ that if $\CL$ is the loop of $\Gamma$ which surrounds $z$ then $\Xi(\CL) = 0$ a.s. By conformal covariance, it suffices to prove this in the case that $z=0$.

We will deduce the result by taking a limit $\kappa \to 4$.  Fix $\kappa_0 \in (8/3,4)$ and $\kappa \in [\kappa_0,4)$.  Let $\eta^\kappa$ be an $\SLE_\kappa^0(\kappa-6)$ process in $\D$ from $-i$ to $0$.  Let $(W^\kappa,O^\kappa)$ be the driving pair for $\eta^\kappa$ (parameterized by log conformal radius as seen from $0$).  Fix $\delta > 0$, let
\[ \tau_1^\kappa = \inf\{t \geq 0 : |W_t^\kappa - O_t^\kappa| \geq \delta \}\]
and let $\sigma_1^\kappa = \inf\{t \geq \tau_1^\kappa : W_t^\kappa = O_t^\kappa\}$.  Given that we have defined $\tau_1^\kappa,\sigma_1^\kappa,\ldots,\tau_k^\kappa,\sigma_k^\kappa$, we let
\[ \tau_{k+1}^\kappa = \inf\{t \geq \sigma_k^\kappa : |W_t^\kappa - O_t^\kappa| \geq \delta \}\]
and let $\sigma_{k+1}^\kappa = \inf\{t \geq \tau_{k+1}^\kappa : W_t^\kappa = O_t^\kappa\}$.  For each $k \in \N$ we have that the law of $\eta^\kappa|_{[\tau_k^\kappa,\sigma_k^\kappa]}$ given $\eta^\kappa|_{[0,\tau_k^\kappa]}$ is that of an $\SLE_\kappa$ process in the component of $\D \setminus \eta^\kappa([0,\tau_k^\kappa])$ with $\eta^\kappa(\tau_k^\kappa)$ on its boundary.

For each $t \geq 0$ we let $\F_t^\kappa = \sigma(\eta^\kappa(s): s \leq t)$.  For each $n \in \N$ we let $\CS_n$ be the set of closed squares $S \subseteq \D$ with side length $2^{-n}$ and with corners in $(2^{-n} \Z)^2$.  Suppose that $S \in \CS_n$.  Fix $\zeta > 0$.  On the event that $S$ has distance at least $\zeta$ from $\eta^\kappa([0,\tau_k^\kappa])$, the probability that $\eta^\kappa|_{[\tau_k^\kappa,\sigma_k^\kappa]}$ hits $S' \in \CS_n$ which is adjacent to $S$ is $O(2^{-\alpha n})$ where $\alpha = 1-\kappa/8$ and the implicit constant depends only on $\zeta$ (see \cite[Lemma~2.10]{lw2013multipoint} and note that the constant in the statement can be taken to be uniform for $\kappa \in [\kappa_0,4]$).  On $\eta^\kappa([0,\sigma_k^\kappa]) \cap S = \emptyset$ but $\eta^\kappa([0,\sigma_k^\kappa])$ hits $S' \in \CS_n$ which is adjacent to $S$ we have that $\E[ \Xi^\kappa(S) \giv \F_{\sigma_k^\kappa}^\kappa] = O(2^{-d n})$.  Altogether, this implies that given $\F_{\tau_k^\kappa}^\kappa$ the conditional expectation of the $\Xi^\kappa$-mass of those $S \in \CS_n$ which are not hit by $\eta^\kappa|_{[0,\sigma_k^\kappa]}$ but are adjacent to $S' \in \CS_n$ which are hit by $\eta^\kappa|_{[0,\sigma_k^\kappa]}$ and have distance at least $\zeta$ to $\eta^\kappa([0,\tau_k^\kappa])$ is $O(2^{(2-\alpha-d)n})$.  By summing over $n$, we see that the conditional expectation of the $\Xi^\kappa$-mass of the set $U_{n,\kappa}^{\delta,\zeta}$ of points which are within distance at most $2^{-n}$ of $\eta^\kappa([\tau_k^\kappa,\sigma_k^\kappa])$ but distance at least $\zeta$ to $\eta^\kappa([0,\tau_k^\kappa])$ is $O(2^{(2-\alpha-d)n})$.  Note that $2-\alpha-d < 0$.

We assume that $\eta^\kappa$ is an exploration of a $\CLE_\kappa$ process $\Gamma^\kappa$ and let $\CL^\kappa$ be the loop of $\Gamma^\kappa$ which is surrounded by the origin.  Let $K^\kappa$ be such that $\eta^\kappa$ is drawing $\CL^\kappa$ at the time $\tau_{K^\kappa}^\kappa$.  We note that $K^\kappa$ is stochastically dominated by a geometric random variable whose parameter is uniform in $\kappa$ (as $\kappa$ is bounded away from $8/3$).  Then we have shown that
\begin{equation}
\label{eqn:probability_converges_to_zero}
\p[ \Xi^\kappa(U_{n,\kappa}^{\delta,\zeta}) \geq \epsilon] \to 0 \quad\text{as}\quad n \to \infty
\end{equation}
uniformly in $\kappa \in [\kappa_0,4)$. Suppose that we have a sequence $(\kappa_j)$ in $[\kappa_0,4)$ with $\kappa_j \to 4$ as $j \to \infty$ and we have used the Skorokhod representation theorem for weak convergence to couple the laws together so that we a.s.\ have $(\eta^{\kappa_j}, \Gamma^{\kappa_j},\Xi^{\kappa_j}) \to (\eta^4,\Gamma^4,\Xi^4)$ as $j \to \infty$.  Then we have that $U_{n,4}^{\delta,\zeta}$ is a.s.\ contained in $U_{n-1,\kappa_j}^{\delta,\zeta}$ for all $j$ large enough.  Thus if it were true that the $\Xi^4$ measure of the part of $\CL^4$ with distance at least $\zeta$ from $\eta^4([0,\tau_{K^4}^4])$ is positive with positive probability we would get a contradiction to~\eqref{eqn:probability_converges_to_zero}.  By sending $\zeta \to 0$ this implies that $\Xi^4( \CL^4 \setminus \eta^4([0,\tau_{K^4}^4])) = 0$ a.s.  Let $z^4$ be the place where $\eta^4$ first hits $\CL^4$.  By sending $\delta \to 0$ we thus see that $\Xi^4(\CL^4 \setminus \{z^4\}) = 0$ a.s.

Suppose that we are on the event that $\eta^4$ disconnects $0$ from $i$ at the time $\sigma_{K^4}^4$.  Then we can couple~$\eta^4$ together with the branch~$\wt{\eta}^4$ of the exploration tree from $-i$ to $i$ used to generate~$\Gamma^4$ to be the same up until the time~$\sigma_{K^4}^4$.  On this event, we note that $z^4$ is the same as the place where the trunk of $\wt{\eta}^4$ first hits $\CL^4$ which is a.s.\ \emph{not} the same as the first place where the time-reversal of the trunk of $\wt{\eta}^4$ first hits $\CL^4$.  Since the process obtained by following the loops of $\Gamma^4$ hit by the time-reversal of the trunk of $\wt{\eta}^4$ also has the law of an $\SLE_4^0(-2)$ we altogether see that $\Xi^4(\CL^4) = 0$ a.s.\ on the aforementioned event.  If $\eta^4$ disconnects $0$ from $i$ before $\sigma_{K^4}^4$, then we can apply the same argument in the complementary component containing $0$ at this time.  Applying this repeatedly if necessary, we see that $\Xi^4(\CL^4) = 0$ a.s.
\end{proof}

We now turn to the proof of Proposition~\ref{prop:CLE_unique}.
\begin{proof}[Proof of Proposition~\ref{prop:CLE_unique}]
Fix $\kappa \in (8/3,8)$. Let $(\p_D)$ and $(\wt{\p}_D)$ be as in the statement and assume that $(\Gamma,\Xi) \sim \p_{\D}$ and $(\Gamma,\Xi^*) \sim \wt{\p}_{\D}$ are coupled so that the $\CLE_\kappa$ is the same in both pairs and normalize them so that $\E[\Xi](\D) = \E[\Xi^*](\D) = 1$.

We note that $\E[\Xi](L) = 0$ and $\E[\Xi^*](L) = 0$ for all but at most countably many horizontal lines $L$ (for otherwise $\E[\Xi](\D) = \infty$ or $\E[\Xi^*](\D) = \infty$).  We can therefore find a countable collection of lines $(L_k)$ so that $\E[\Xi](L_k) = 0$ and $\E[\Xi^*](L_k) = 0$ for all $k$ and so that $\D \cap (\cup_k L_k)$ is dense in $\D$.  For each $n$, we let $\F_n$ be the $\sigma$-algebra which is generated by the loops of $\Gamma$ which intersect $\cup_{k=1}^n L_k$.  By density, we have that for every loop $\CL$ of $\Gamma$ there exists $k$ so that $L_k$ intersects the region surrounded by $\CL$ which, in turn, implies that $\F_\infty = \sigma(\F_n : n \in \N)$ is equal to $\sigma(\Gamma)$.  For each $n$, we let $D_j^n$ denote the collection of components in the complement in $\D$ of the closure of the union of the points surrounded by loops of $\Gamma$ which intersect $\cup_{k=1}^n L_k$.  For each $j,n$, we let $\varphi_j^n: \D \to D_j^n$ be a conformal map with some arbitrary rule for fixing it (for example by considering some fixed enumeration of the rationals $(r_\ell)$ in $\D$, taking $r_{j,n}$ to be the element of $(r_\ell)$ in $D_j^n$ with the smallest index $\ell$, and then taking $\varphi_j^n$ to be the unique conformal map from $\D$ to $D_j^n$ with $\varphi_j^n(0) = r_{j,n}$ and $(\varphi_j^n)'(0) > 0$).  Then the conditional law given $\F_n$ of the image under $(\varphi_j^n)^{-1}$ of the loops of $\Gamma$ contained in $\ol{D_j^n}$  is that of a $\CLE_\kappa$ in $\D$ independently of the loops of $\Gamma$ contained in the other $\ol{D_i^n}$. For each $n \in \N$, we define the random measure $\Xi_n$ by
\begin{align*}
    \Xi_n(O) = \E[\Xi \giv \F_n](O).
\end{align*}
Then for each $n \in \N$
\begin{align*}
    \E[\Xi_n(\D)] = \E[\Xi_n](\D) = \E[\Xi](\D) = 1.
\end{align*}
Thus by the martingale convergence theorem for every open set $O \subseteq \D$ we have that $\Xi_n(O)$ converges a.s.\ to $\Xi(O)$ as $\Xi(O)$ is $\F_\infty$-measurable.

Let $(\Xi^{j,n})_j$ be i.i.d.\ copies of $\Xi$. Then by Lemmas~\ref{lem:density}, \ref{lem:condint} we have that
\begin{align*}
    &\Xi_n( \cup_j D_j^n) = \E\big[\Xi(\cup_j D_j^n)\, \big|\, \F_n\big] = \sum_j \E\big[ \Xi(D_j^n) \, \big| \, \F_n \big] = \sum_j \E\!\left[ \int_{\D} |(\varphi_j^n)'(z)|^d d\Xi^{j,n}(z) \, \middle| \, \F_n \right] \\
    &= \sum_j \int_{\D} |(\varphi_j^n)'(z)|^d d\E[\Xi^{j,n}](z) = C_\Xi \sum_j \int_{\D} |(\varphi_j^n)'(z)|^d \left( \frac{1}{1-|z|^2} \right)^{2-d} dz.
\end{align*}
The same calculation implies that if we let
\begin{align*}
    \Xi_n^*(O) = \E[\Xi_n^*(O) \giv \F_n]
\end{align*}
then $\Xi_n^*(\cup_j D_j^n) = \Xi_n(\cup_j D_j^n)$ a.s. (The reason that we do not have an extra multiplicative factor is that we normalized $\Xi^*$, $\Xi$ so that $\E[\Xi(\D)] = \E[\Xi^*(\D)] = 1$ in the beginning.)

Lemma~\ref{lem:loop_has_zero_mass} and our choice of lines $(L_k)$ implies that $\E[ \Xi( \D \setminus \cup_j D_j^n) \giv \F_n] = 0$ for each $n \in \N$.  We claim that it also holds that $\E[ \Xi^*( \D \setminus \cup_j D_j^n) \giv \F_n] = 0$ for all $n \in \N$. Indeed for each $n$ we have that
\begin{align*}
    1 = \E[ \Xi(\D) ] = \E[\Xi_n(\cup_j D_j^n)] = \E[\Xi_n^*(\cup_j D_j^n)].
\end{align*}
Thus if $\Xi^*(\D \setminus \cup_j D_j^n)$ is not a.s.\ zero then $\E[\Xi^*(\D)] > 1$ which contradicts our initial choice of normalization.  Thus $\Xi_n = \Xi_n^*$ a.s.  Consequently, since $\Xi_n \to \Xi$ and $\Xi_n^* \to \Xi^*$ and $\Xi_n = \Xi_n^*$ for all $n \in \N$ a.s., we have that $\Xi^* = \Xi$ a.s., proving the proposition.
\end{proof}

We now wrap up the proof of Theorem~\ref{thm:mainresult} by collecting the results of the last two sections.
\begin{proof}[Proof of Theorem~\ref{thm:mainresult}]
By Lemmas~\ref{lem:conformal_covariance_carpet} and~\ref{lem:conformal_covariance_gasket} and Proposition~\ref{prop:CLE4_measure} we have that there exist laws on pairs $(\Gamma,\Xi)$ as in Definition~\ref{def:cov_measure} with $d = d(\kappa)$ given by~\eqref{eq:CLEdim} in the cases $\kappa \in (8/3,4)$, $\kappa \in (4,8)$ and $\kappa = 4$, respectively (note that property \eqref{it:locally_finite_expectation} follows from Lemmas~\ref{lem:locally_finite_carpet} and~\ref{lem:locally_finite_gasket} in the cases $\kappa \in (0,4)$ and $\kappa \in (4,8)$ and by the choice of normalization in the proof of Proposition~\ref{prop:CLE4_measure} in the case $\kappa = 4$). The uniqueness of the law on such pairs follows by Proposition~\ref{prop:CLE_unique}. Hence, the result follows.
\end{proof}

\bibliographystyle{abbrv}
\bibliography{biblio}

\begin{thebibliography}{10}

\bibitem{v2008dim}
V.~Beffara.
\newblock The dimension of the {SLE} curves.
\newblock {\em Ann. Probab.}, 36(4):1421--1452, 2008.

\bibitem{benoist2018natural}
S.~Benoist.
\newblock Natural parametrization of {SLE}: the {G}aussian free field point of
  view.
\newblock {\em Electron. J. Probab.}, 23:Paper No. 103, 16, 2018.

\bibitem{bh2019ising}
S.~Benoist and C.~Hongler.
\newblock The scaling limit of critical {I}sing interfaces is {CLE$_3$}.
\newblock {\em Ann. Probab.}, 47(4):2049--2086, 2019.

\bibitem{bertoin1996book}
J.~Bertoin.
\newblock {\em L\'{e}vy processes}, volume 121 of {\em Cambridge Tracts in
  Mathematics}.
\newblock Cambridge University Press, Cambridge, 1996.

\bibitem{cn2006fullpercolation}
F.~Camia and C.~M. Newman.
\newblock Two-dimensional critical percolation: the full scaling limit.
\newblock {\em Comm. Math. Phys.}, 268(1):1--38, 2006.

\bibitem{ck2014looptrees}
N.~Curien and I.~Kortchemski.
\newblock Random stable looptrees.
\newblock {\em Electron. J. Probab.}, 19:no. 108, 35, 2014.

\bibitem{dddf2020tightness}
J.~Ding, J.~Dub\'{e}dat, A.~Dunlap, and H.~Falconet.
\newblock Tightness of {L}iouville first passage percolation for {$\gamma \in
  (0,2)$}.
\newblock {\em Publ. Math. Inst. Hautes \'{E}tudes Sci.}, 132:353--403, 2020.

\bibitem{dms2014mating}
B.~{Duplantier}, J.~{Miller}, and S.~{Sheffield}.
\newblock {Liouville quantum gravity as a mating of trees}.
\newblock {\em arXiv e-prints}, page arXiv:1409.7055, Sept. 2014.
\newblock To appear in Asterisque.

\bibitem{ds2011kpz}
B.~Duplantier and S.~Sheffield.
\newblock Liouville quantum gravity and {KPZ}.
\newblock {\em Invent. Math.}, 185(2):333--393, 2011.

\bibitem{emot1953functions2}
A.~Erd\'{e}lyi, W.~Magnus, F.~Oberhettinger, and F.~G. Tricomi.
\newblock {\em Higher transcendental functions. {V}ol. {II}}.
\newblock Robert E. Krieger Publishing Co., Inc., Melbourne, Fla., 1981.
\newblock Based on notes left by Harry Bateman, Reprint of the 1953 original.

\bibitem{gm2005harmonic}
J.~B. Garnett and D.~E. Marshall.
\newblock {\em Harmonic measure}, volume~2 of {\em New Mathematical
  Monographs}.
\newblock Cambridge University Press, Cambridge, 2005.

\bibitem{gm2021metric}
E.~Gwynne and J.~Miller.
\newblock Existence and uniqueness of the {L}iouville quantum gravity metric
  for {$\gamma\in(0,2)$}.
\newblock {\em Invent. Math.}, 223(1):213--333, 2021.

\bibitem{gmq2021inversion}
E.~Gwynne, J.~Miller, and W.~Qian.
\newblock Conformal {I}nvariance of {${\rm CLE}_k$} on the {R}iemann {S}phere
  for {$\kappa\in(4, 8)$}.
\newblock {\em Int. Math. Res. Not. IMRN}, (23):17971--18036, 2021.

\bibitem{kallenberg2017rmbook}
O.~Kallenberg.
\newblock {\em Random measures, theory and applications}, volume~77 of {\em
  Probability Theory and Stochastic Modelling}.
\newblock Springer, Cham, 2017.

\bibitem{kms2021regularity}
K.~Kavvadias, J.~Miller, and L.~Schoug.
\newblock {Regularity of the $\mathrm{SLE}_4$ uniformizing map and the
  $\mathrm{SLE}_8$ trace}.
\newblock {\em arXiv e-prints}, page arXiv:2107.03365, July 2021.

\bibitem{ks2019ising}
A.~Kemppainen and S.~Smirnov.
\newblock Conformal invariance of boundary touching loops of {FK} {I}sing
  model.
\newblock {\em Comm. Math. Phys.}, 369(1):49--98, 2019.

\bibitem{ll2009geo}
S.~Lalley, G.~Lawler, and H.~Narayanan.
\newblock Geometric interpretation of half-plane capacity.
\newblock {\em Electron. Commun. Probab.}, 14:566--571, 2009.

\bibitem{lsw2003confres}
G.~Lawler, O.~Schramm, and W.~Werner.
\newblock Conformal restriction: the chordal case.
\newblock {\em J. Amer. Math. Soc.}, 16(4):917--955, 2003.

\bibitem{lawler2005slebook}
G.~F. Lawler.
\newblock {\em Conformally invariant processes in the plane}, volume 114 of
  {\em Mathematical Surveys and Monographs}.
\newblock American Mathematical Society, Providence, RI, 2005.

\bibitem{lr2015minkowski}
G.~F. Lawler and M.~A. Rezaei.
\newblock Minkowski content and natural parameterization for the
  {S}chramm-{L}oewner evolution.
\newblock {\em Ann. Probab.}, 43(3):1082--1120, 2015.

\bibitem{lsw2004lerwust}
G.~F. Lawler, O.~Schramm, and W.~Werner.
\newblock Conformal invariance of planar loop-erased random walks and uniform
  spanning trees.
\newblock {\em Ann. Probab.}, 32(1B):939--995, 2004.

\bibitem{ls2011natural}
G.~F. Lawler and S.~Sheffield.
\newblock A natural parametrization for the {S}chramm-{L}oewner evolution.
\newblock {\em Ann. Probab.}, 39(5):1896--1937, 2011.

\bibitem{lv2021natural}
G.~F. Lawler and F.~Viklund.
\newblock Convergence of loop-erased random walk in the natural
  parameterization.
\newblock {\em Duke Math. J.}, 170(10):2289--2370, 2021.

\bibitem{lw2004loopsoup}
G.~F. Lawler and W.~Werner.
\newblock The {B}rownian loop soup.
\newblock {\em Probab. Theory Related Fields}, 128(4):565--588, 2004.

\bibitem{lw2013multipoint}
G.~F. Lawler and B.~M. Werness.
\newblock Multi-point {G}reen's functions for {SLE} and an estimate of
  {B}effara.
\newblock {\em Ann. Probab.}, 41(3A):1513--1555, 2013.

\bibitem{lz2013natural}
G.~F. Lawler and W.~Zhou.
\newblock {\it {SLE}} curves and natural parametrization.
\newblock {\em Ann. Probab.}, 41(3A):1556--1584, 2013.

\bibitem{lupu2019convergence}
T.~Lupu.
\newblock Convergence of the two-dimensional random walk loop-soup clusters to
  {CLE}.
\newblock {\em J. Eur. Math. Soc. (JEMS)}, 21(4):1201--1227, 2019.

\bibitem{ms2016ig1}
J.~Miller and S.~Sheffield.
\newblock Imaginary geometry {I}: interacting {SLE}s.
\newblock {\em Probab. Theory Related Fields}, 164(3-4):553--705, 2016.

\bibitem{ms2016ig2}
J.~Miller and S.~Sheffield.
\newblock Imaginary geometry {II}: reversibility of {${\rm
  SLE}_\kappa(\rho_1;\rho_2)$} for {$\kappa\in(0,4)$}.
\newblock {\em Ann. Probab.}, 44(3):1647--1722, 2016.

\bibitem{ms2016ig3}
J.~Miller and S.~Sheffield.
\newblock Imaginary geometry {III}: reversibility of {$\rm SLE_\kappa$} for
  {$\kappa\in(4,8)$}.
\newblock {\em Ann. of Math. (2)}, 184(2):455--486, 2016.

\bibitem{ms2017ig4}
J.~Miller and S.~Sheffield.
\newblock Imaginary geometry {IV}: interior rays, whole-plane reversibility,
  and space-filling trees.
\newblock {\em Probab. Theory Related Fields}, 169(3-4):729--869, 2017.

\bibitem{ms2019lightcones}
J.~Miller and S.~Sheffield.
\newblock Gaussian free field light cones and {${\rm SLE}_\kappa(\rho)$}.
\newblock {\em Ann. Probab.}, 47(6):3606--3648, 2019.

\bibitem{ms2020lqgtbm1}
J.~Miller and S.~Sheffield.
\newblock Liouville quantum gravity and the {B}rownian map {I}: the {${\rm
  QLE}(8/3,0)$} metric.
\newblock {\em Invent. Math.}, 219(1):75--152, 2020.

\bibitem{ms2021lqgtbm2}
J.~Miller and S.~Sheffield.
\newblock Liouville quantum gravity and the {B}rownian map {II}: {G}eodesics
  and continuity of the embedding.
\newblock {\em Ann. Probab.}, 49(6):2732--2829, 2021.

\bibitem{msw2017clepercolation}
J.~Miller, S.~Sheffield, and W.~Werner.
\newblock C{LE} percolations.
\newblock {\em Forum Math. Pi}, 5:e4, 102, 2017.

\bibitem{msw2020simpleclelqg}
J.~Miller, S.~Sheffield, and W.~Werner.
\newblock {Simple conformal loop ensembles on {L}iouville quantum gravity}.
\newblock {\em arXiv e-prints}, page arXiv:2002.05698, Feb. 2020.
\newblock To appear in Annals of Probability.

\bibitem{msw2020nonsimpleclelqg}
J.~Miller, S.~Sheffield, and W.~Werner.
\newblock Non-simple conformal loop ensembles on {L}iouville quantum gravity
  and the law of {CLE} percolation interfaces.
\newblock {\em Probab. Theory Related Fields}, 181(1-3):669--710, 2021.

\bibitem{msw2014gasketdim}
J.~Miller, N.~Sun, and D.~B. Wilson.
\newblock The {H}ausdorff dimension of the {CLE} gasket.
\newblock {\em Ann. Probab.}, 42(4):1644--1665, 2014.

\bibitem{mp2010brownian}
P.~M\"{o}rters and Y.~Peres.
\newblock {\em Brownian motion}, volume~30 of {\em Cambridge Series in
  Statistical and Probabilistic Mathematics}.
\newblock Cambridge University Press, Cambridge, 2010.
\newblock With an appendix by Oded Schramm and Wendelin Werner.

\bibitem{nw2011soupsdim}
{\c{S}}.~Nacu and W.~Werner.
\newblock Random soups, carpets and fractal dimensions.
\newblock {\em J. Lond. Math. Soc. (2)}, 83(3):789--809, 2011.

\bibitem{rs2005basic}
S.~Rohde and O.~Schramm.
\newblock Basic properties of {SLE}.
\newblock {\em Ann. of Math. (2)}, 161(2):883--924, 2005.

\bibitem{ssv2020tvsdim}
L.~Schoug, A.~Sepúlveda, and F.~Viklund.
\newblock {Dimensions of Two-Valued Sets via Imaginary Chaos}.
\newblock {\em International Mathematics Research Notices}, 10 2020.
\newblock rnaa250.

\bibitem{s2000sle}
O.~Schramm.
\newblock Scaling limits of loop-erased random walks and uniform spanning
  trees.
\newblock {\em Israel J. Math.}, 118:221--288, 2000.

\bibitem{ssw2009confrad}
O.~Schramm, S.~Sheffield, and D.~B. Wilson.
\newblock Conformal radii for conformal loop ensembles.
\newblock {\em Comm. Math. Phys.}, 288(1):43--53, 2009.

\bibitem{sw2005coordinate}
O.~Schramm and D.~B. Wilson.
\newblock S{LE} coordinate changes.
\newblock {\em New York J. Math.}, 11:659--669, 2005.

\bibitem{sheffield2007gff}
S.~Sheffield.
\newblock Gaussian free fields for mathematicians.
\newblock {\em Probab. Theory Related Fields}, 139(3-4):521--541, 2007.

\bibitem{sheffield2009exploration}
S.~Sheffield.
\newblock Exploration trees and conformal loop ensembles.
\newblock {\em Duke Math. J.}, 147(1):79--129, 2009.

\bibitem{sheffield2016zipper}
S.~Sheffield.
\newblock Conformal weldings of random surfaces: {SLE} and the quantum gravity
  zipper.
\newblock {\em Ann. Probab.}, 44(5):3474--3545, 2016.

\bibitem{sw2012cle}
S.~Sheffield and W.~Werner.
\newblock Conformal loop ensembles: the {M}arkovian characterization and the
  loop-soup construction.
\newblock {\em Ann. of Math. (2)}, 176(3):1827--1917, 2012.

\bibitem{s2001perc}
S.~Smirnov.
\newblock Critical percolation in the plane: conformal invariance, {C}ardy's
  formula, scaling limits.
\newblock {\em C. R. Acad. Sci. Paris S\'{e}r. I Math.}, 333(3):239--244, 2001.

\bibitem{smirnov2001conformal}
S.~Smirnov.
\newblock Critical percolation in the plane: conformal invariance, {C}ardy's
  formula, scaling limits.
\newblock {\em C. R. Acad. Sci. Paris S\'{e}r. I Math.}, 333(3):239--244, 2001.

\bibitem{s2010ising}
S.~Smirnov.
\newblock Conformal invariance in random cluster models. {I}. {H}olomorphic
  fermions in the {I}sing model.
\newblock {\em Ann. of Math. (2)}, 172(2):1435--1467, 2010.

\bibitem{wz2017boundaryarm}
H.~Wu and D.~Zhan.
\newblock Boundary arm exponents for {SLE}.
\newblock {\em Electron. J. Probab.}, 22:Paper No. 89, 26, 2017.

\bibitem{zhan2016ergodicity}
D.~Zhan.
\newblock Ergodicity of the tip of an {SLE} curve.
\newblock {\em Probab. Theory Related Fields}, 164(1-2):333--360, 2016.

\bibitem{zhan2021loop}
D.~Zhan.
\newblock S{LE} loop measures.
\newblock {\em Probab. Theory Related Fields}, 179(1-2):345--406, 2021.

\end{thebibliography}

\end{document}